\documentclass[12pt]{amsart}
\usepackage[margin=0.81in]{geometry}

\usepackage{amsmath, amsthm, amssymb, amsfonts, amsrefs}

\usepackage[usenames,dvipsnames,svgnames,table]{xcolor}

\newtheorem{theorem}{Theorem}[section]
\newtheorem{corollary}[theorem]{Corollary}
\newtheorem{lemma}[theorem]{Lemma}
\newtheorem{proposition}[theorem]{Proposition}

\theoremstyle{definition}

\theoremstyle{remark}
\newtheorem{rem}[theorem]{Remark}

\newcommand{\R}{{\mathbb R}}
\newcommand{\N}{{\mathbb N}}

\renewcommand{\leq}{\leqslant}
\renewcommand{\le}{\leqslant}

\renewcommand{\ge}{\geqslant}

\numberwithin{equation}{section}

\newcommand{\Per}{{\rm{Per}}}
\newcommand{\PPer}{{\rm{Per}}^\star}

\begin{document}

\author{Serena Dipierro}
\address{School of Mathematics and Statistics,
University of Melbourne,
813 Swanston Street, Parkville VIC 3010, Australia, and
Dipartimento di Matematica, Universit\`a degli studi di Milano,
Via Saldini 50, 20133 Milan, Italy}
\email{s.dipierro@unimelb.edu.au}

\author{Aram Karakhanyan}
\address{Maxwell Institute for
Mathematical Sciences and School of Mathematics, University of
Edinburgh, James Clerk Maxwell Building, Peter Guthrie Tait Road,
Edinburgh EH9 3FD, United Kingdom.}
\email{aram.karakhanyan@ed.ac.uk}

\author{Enrico Valdinoci}
\address{School of Mathematics and Statistics,
University of Melbourne,
813 Swanston Street, Parkville VIC 3010, Australia, and
Istituto di Matematica Applicata e Tecnologie 
Informatiche, Consiglio Nazionale delle Ricerche, Via Ferrata 1, 27100
Pavia, Italy,
and Dipartimento di Matematica, Universit\`a degli studi di Milano,
Via Saldini 50, 20133 Milan, Italy.}
\email{enrico@math.utexas.edu}

\title{A class of unstable free boundary problems}

\begin{abstract}
We consider the free boundary problem
arising from an energy functional which is
the sum of a Dirichlet energy and a nonlinear
function of either the classical or the fractional perimeter.

The main difference with the existing literature is that
the total energy is here a nonlinear superposition
of the either local or nonlocal surface tension effect
with the elastic energy.

In sharp contrast with the linear case, the problem
considered in this paper is unstable, namely a minimizer
in a given domain is not necessarily a minimizer in a smaller domain.

We provide an explicit example for this instability.
We also give a free boundary condition, which
emphasizes the role played by the domain in the geometry
of the free boundary. In addition, we provide density
estimates for the free boundary and regularity results for the minimal solution.

As far as we know, this is the first case in which a nonlinear
function of the 
perimeter is studied in this type of problems.
Also, the results obtained in this nonlinear
setting are new even in the case of the local perimeter, and indeed
the instability feature is not a consequence of the possibly nonlocality
of the problem, but it is due to the nonlinear character of the energy functional.
\end{abstract}
\maketitle

\section{Introduction}

In this paper we consider a free boundary problem
given by the superposition of a Dirichlet energy
and an either classical or nonlocal perimeter functional. Differently from
the existing literature, here we take into account the possibility
that this energy superposition occurs in a nonlinear way,
that is the total energy functional is the sum of the Dirichlet energy plus
a nonlinear function of the either local or nonlocal perimeter of
the interface.\medskip

Unlike the cases already present in the literature,
the nonlinear problem that we study may present a structural
instability induced by the domain, namely a minimizer
in a large domain may fail to be a minimizer in a small domain.
This fact prevents the use of the scaling arguments, which
are frequently exploited in classical free boundary problems.\medskip

In this paper, after providing an explicit example of this type of
structural instability, we describe the free boundary equation,
which also underlines the striking role played by the total
(either local or nonlocal) perimeter of the minimizing set in the domain,
as modulated by the nonlinearity, 
in the local geometry of the interface.
Then, we will present results concerning 
the H\"older regularity of the minimal solutions
and the density of the interfaces in the one-phase problem.
\medskip

The mathematical setting in which we work is the following.
Given an (open, Lipschitz and bounded) 
domain~$\Omega\subset\R^n$ and~$\sigma\in(0,1]$,
we use the notation~$\Per_\sigma(E,\Omega)$ for the classical
perimeter of~$E$ in~$\Omega$
when~$\sigma=1$ (which will be often denoted as~$\Per(E,\Omega)$,
see e.g.~\cites{AFP, maggi})
and the fractional perimeter of~$E$ in~$\Omega$
when~$\sigma\in(0,1)$ (see~\cite{CRS}).
More explicitly, if~$\sigma\in(0,1)$, we have that
\begin{equation}\label{9g12:LA} 
\Per_\sigma (E,\Omega):=
L(E\cap \Omega,E^c) + L(E^c\cap\Omega,E\cap\Omega^c),\end{equation}
where, for any measurable subsets~$A$, $B\subseteq \R^n$
with~$A\cap B$ of measure zero, we set
$$ L(A,B) := \iint_{A\times B}\frac{dx\,dy}{|x-y|^{n+\sigma}}.$$
As customary, we are using here the superscript~$c$
for complementary set, i.e.~$E^c:=\R^n\setminus E$.

The notation used for~$\Per_\sigma$ when~$\sigma=1$
is inspired by the fact that~$\Per_\sigma$, suitably rescaled,
approaches the classical perimeter as~$s\nearrow1$,
see e.g.~\cites{BoBrM, DAV, CVCV, ADMa}.

In our framework, the role played by the fractional
perimeter is to allow long-range interaction
to contribute to the energy arising from surface tension
and phase segregation.

As a matter of fact, the fractional perimeter~$\Per_\sigma$
naturally arises when one considers phase transition models
with long-range particle interactions (see e.g.~\cite{MR3133422}):
roughly speaking, in this type of models, the remote
interactions of the particles are sufficently strong to persist
even at a large scale, by possibly modifying the behavior of the
phase separation. 

The fractional perimeter~$\Per_\sigma$ has also natural
applications in motion by nonlocal mean curvatures,
which in turn arises naturally in the study of cellular automata
and in the image digitalization procedures (see e.g.~\cite{imbert}).
\medskip

It is also convenient\footnote{The explicit value of~$\Upsilon$
plays no major role, since it can be fixed by an ``initial
scaling'' of the problem, but we decided to require it to be
less than~$\frac{1}{100}$ to emphasize, from the psychological point of view,
that~$\Omega_\Upsilon$ can be thought as a small enlargement of~$\Omega$.

The reason for which we introduced such~$\Upsilon$ is that, in the classical
case, the interfaces inside~$\Omega$ do not see the contributions
that may come along~$\partial\Omega$, since~$\Omega$ is taken to be open
(viceversa, in the nonlocal case, these contributions are always counted).
By enlarging the domain~$\Omega$ by a small quantity~$\Upsilon$,
we are able to count also the contributions on~$\partial\Omega$
and this, roughly speaking, boils down to computing the classical perimeter
in the closure of~$\Omega$.}
to fix~$\Upsilon\in \left(0,\frac{1}{100}\right]$
and set
\begin{equation}\label{PPP:D}
\begin{split}
&\Omega_\Upsilon :=\bigcup_{p\in\Omega} B_\Upsilon (p)\\
{\mbox{and }}\qquad&\PPer_\sigma (E,\Omega) =
\left\{ \begin{matrix}\Per(E,\Omega_\Upsilon) & {\mbox{if $\sigma=1$,}}\\
\Per_\sigma(E,\Omega) & {\mbox{if $\sigma\in(0,1)$}}. \end{matrix}
\right.
\end{split}
\end{equation}
We consider a monotone nondecreasing and lower semicontinuous
function~$\Phi:[0,+\infty)\to [0,+\infty)$, with
\begin{equation}\label{COE}
\lim_{t\to+\infty} \Phi(t)=+\infty.
\end{equation}
For any measurable function~$u:\R^n\to\R$,
such that~$|\nabla u|\in L^2(\Omega)$ and any measurable
subset~$E\subseteq\R^n$
such that~$u\ge0$ a.e. in~$E$ and~$u\le0$ a.e. in~$E^c$, we consider
the energy functional
\begin{equation}\label{FUNCT}
{\mathcal{E}}_\Omega(u,E):=
\int_\Omega |\nabla u(x)|^2\,dx + \Phi\Big( \PPer_\sigma (E,\Omega)\Big)
.\end{equation}
As usual, 
the notation~$\nabla u$ stands for the distributional
gradient. 

When~$\Phi$ is the identity, the functional in~\eqref{FUNCT}
provides a typical problem for (either local or nonlocal)
free boundary problems, see~\cites{salsa, CSV}.

The goal of this
paper is to study the minimizers
of the functional in~\eqref{FUNCT}.
For this, we say that~$(u,E)$ is an admissible pair
if:
\begin{itemize}
\item $u:\R^n\to\R$ is a
measurable function
such that~$u\in H^1(\Omega)$,
\item $E\subseteq\R^n$ is a measurable set with~$\PPer_\sigma(E,\Omega)<+
\infty$, and
\item $u\ge0$ a.e. in~$E$ and~$u\le0$ a.e. in~$E^c$.\end{itemize}
Then, we say that~$(u,E)$ is a minimal pair in~$\Omega$ if
\begin{itemize}
\item $(u,E)$ is an admissible pair,
\item ${\mathcal{E}}_\Omega(u,E)<+\infty$, and
\item for any admissible pair~$(v,F)$ such that~$v-u\in H^1_0(\Omega)$
and~$F\setminus\Omega= E\setminus \Omega$ 
up to sets of measure zero, we have that
$$ {\mathcal{E}}_\Omega(u,E)\le {\mathcal{E}}_\Omega(v,F).$$
\end{itemize}
The existence\footnote{As a technical remark,
we point out that the definition in~\eqref{PPP:D}
is useful to make sense of nontrivial versions
of this minimization problem when~$\sigma=1$ and~$u\ge0$.
Indeed, in this case, the setting
in~\eqref{PPP:D} ``forces'' the sets to interact with the
boundary data. This expedient is not necessary when~$\sigma=0$
since, in this case, the nonlocal effect produces the nontrivial
interactions.} 
of minimal pairs for fixed domains
and fixed conditions outside the domain follows from
the direct methods in the calculus of variations (see Lemma~\ref{EX}
below for details).
\medskip

A natural question in this framework is whether or not
this minimization procedure is ``stable'' with respect to
the choice of the domain, i.e. whether or not
a minimal pair in a domain~$\Omega$ is also
a minimal pair in any subdomain~$\Omega'\subset\Omega$. 
This stability property is indeed typical
for ``linear'' free boundary problems, i.e. when~$\Phi$
is the identity, see~\cites{salsa, CSV},
and it often plays a crucial role in many arguments
based on scaling and blow-up analysis.

In the ``nonlinear'' case, i.e. when~$\Phi$ is not the identity,
this stability property is lost,
and we will provide a concrete example for that.
In further detail, we 
consider the planar case of~$\R^2$,
we take coordinates~$X:=(x,y)\in\R^2$ and we set
\begin{equation}\label{tilde u}
\tilde u(x,y):= xy \end{equation}
and
\begin{equation}\label{tilde E}
\begin{split}
\tilde E\,&:=\{(x,y)\in\R^2
{\mbox{ s.t. }} xy>0\}\\ &=
\{(x,y)\in\R^2 {\mbox{ s.t. }} x>0 {\mbox{ and }}y>0\}
\cup
\{(x,y)\in\R^2 {\mbox{ s.t. }} x<0 {\mbox{ and }}y<0\}.\end{split}\end{equation}
In this setting, we
show that:

\begin{theorem}[An explicit counterexample]\label{EXAMPLE}
There exists~$K_o>2$ such that the following statement is true.
Let~$n=2$.
Assume that
$$ {\mbox{$\Phi(t) = t^\gamma$ for any~$t\in[0,1]$}}$$
for some
$$ \gamma\in\left(0,\frac{4}{2-\sigma}\right),$$
and
\begin{equation} \label{ART}
{\mbox{$\Phi(t)=1$ for any~$t\in[ 2,K_o]$. }}\end{equation}
Then, there exist~$R_o>r_o>0$ such that~$(\tilde u,\tilde E)$
is a minimal pair in~$B_{R_o}$
and is not a minimal pair in~$B_r$ for any~$r\in(0,r_o]$.
\end{theorem}

The heuristic idea underneath Theorem~\ref{EXAMPLE} is, roughly speaking,
that the nonlinear energy term~$\Phi$ weights differently
the fractional perimeter with respect to the Dirichlet energy
in different energy regimes, so it may favor a minimal pair~$(u,E)$
to be either ``close to a harmonic function'' in the~$u$
or ``close to a fractional minimal surface'' in the~$E$,
depending on the minimal energy level reached in a given domain.
\medskip

It is worth stressing that, in other circumstances,
rather surprising instability features in interface problems
arise as a consequence of the fractional behavior of the energy,
see for instance~\cite{stick}. Differently from these cases,
the unstable free boundaries presented in Theorem~\ref{EXAMPLE}
are not caused by the existence of possibly nonlocal features,
and indeed Theorem~\ref{EXAMPLE} holds true (and is new) even in the
case of the local perimeter.\medskip

The instability phenomenon 
pointed out by Theorem~\ref{EXAMPLE} in a concrete case
is also quite general, as it can be understood
also in the light of the associated equation on the free
boundary. Indeed, the free boundary equation
takes into account a ``global'' term of the type~$\Phi'\Big(\PPer_\sigma(E,\Omega)\Big)$,
which varies in dependence of the domain~$\Omega$.
To clarify this point, we denote by~$H_{\sigma}^E$
the (either classical or fractional) mean curvature of~$\partial E$
(see~\cites{CRS, AV} for the case~$\sigma\in(0,1)$). Namely, if~$\sigma=1$
the above notation stands for the classical
mean curvature, while if~$\sigma\in(0,1)$, if~$x\in\partial E$, we set
$$ H_{\sigma}^E(x):=\limsup_{\delta\to0}
\int_{\R^n\setminus B_\delta(x)} \frac{\chi_{E^c}(y)-\chi_{E}(y)}{|x-y|^{n+\sigma}}\,dy.$$
In this setting, we have:

\begin{theorem}[Free boundary equation]\label{FBW}
Let~$\Phi\in C^{1,\alpha}(0,+\infty)$, for some~$\alpha\in(0,1)$.
Assume that~$(u,E)$ is a minimal pair in~$\Omega$.
Assume that
\begin{equation}\label{II51}
\begin{split}
&{\mbox{$(\partial E)\cap\Omega$ is of class~$C^{1,\tau}$ 
with~$\tau\in(\sigma,1)$ when~$\sigma\in(0,1)$}}\\
&{\mbox{and of class~$C^2$ when~$\sigma=1$.}}\end{split}\end{equation}
Suppose also
that
\begin{equation}\label{7u994a5}
{\mbox{$u>0$ in the interior of~$E\cap\Omega$,
that~$u<0$ in the interior of~$E^c\cap\Omega$,}}\end{equation}
and that
\begin{equation}\label{II52}
u\in C^1(\overline{\{u>0\}\cap\Omega})\cap
C^1(\overline{\{u<0\}\cap\Omega}).\end{equation}
Let also~$\nu$ be the exterior normal of~$E$,
and for any~$x\in(\partial E)\cap\Omega$ let
\begin{equation}\label{II53}
\partial_{\nu}^+ u (x):=\lim_{t\to0}\frac{u(x-t\nu)-u(x)}{t}
\;{\mbox{ and }}\;
\partial_{\nu}^- u (x):=\lim_{t\to0}\frac{u(x+t\nu)-u(x)}{t}.\end{equation}
Then, for any~$x\in(\partial E)\cap\Omega$, we have
\begin{equation}\label{MCE}
\big( \partial_{\nu}^+ u (x)\big)^2
-\big( \partial_{\nu}^- u (x)\big)^2
= H_{\sigma}^E(x)\,\Phi'\Big(
\PPer_\sigma (E,\Omega)
\Big).\end{equation}
\end{theorem}

We remark that equation~\eqref{MCE} has a simple
geometric consequence when~$\Phi'>0$ and we consider the
one-phase problem in which~$u\ge0$: indeed, in this case,
we have that~$\partial_{\nu}^- u=0$ and therefore
formula~\eqref{MCE} reduces to
$$ \big( \partial_{\nu}^+ u (x)\big)^2
= H_{\sigma}^E(x)\,\Phi'\Big(
\PPer_\sigma (E,\Omega)\Big).$$
In particular, we get that~$H_{\sigma}^E\ge0$, namely,
in this case, the (either classical or fractional) mean curvature of
the free boundary is nonnegative.
\medskip

In order to better understand the structure
of the solution and of the free boundary points, we now focus,
for the sake of simplicity, to the one-phase case,
i.e. we suppose that~$u\ge0$ to start with.
In this setting, 
we investigate the H\"older regularity of the
function~$u$,
by obtaining uniform bounds and uniform growth
conditions from the free boundary. 
For this, it is also convenient to introduce the auxiliary 
set
\begin{equation} \label{DEF:0}
{\mathcal{U}}_0 := \big\{ x\in\Omega {\mbox{ s.t.
there exists a sequence }} x_k\in\Omega {\mbox{ s.t. }}
x_k\to x {\mbox{ with }}
u(x_k)\to 0 {\mbox{ as }}k\to+\infty\big\}.\end{equation}
Notice that~$\{u=0\}$ lies in~${\mathcal{U}}_0$ (just taking
a constant sequence in the definition above).
Also, if~$u\ge0$, then~$\partial E$ lies in~$ {\mathcal{U}}_0$
(since in this case~$u$ must vanish in the complement of~$E$).

Of course, when~$u$ is continuous, such set lies in the zero level set of~$u$,
but since we do not have this information a priori, it is useful to consider
explicitly this set, and prove the following result:

\begin{theorem}[Growth from the free boundary]\label{GFB}
Let~$R_o$, $Q>0$. Assume that
\begin{equation}\label{LIPASS}
{\mbox{$\Phi$ is Lipschitz continuous in~$[0,Q]$,
with Lipschitz constant bounded by~$L_Q$.}}
\end{equation}
Assume that~$(u,E)$ is a minimal pair in~$\Omega$,
with~$B_{R_o}\Subset\Omega$,
\begin{equation}\label{jkLPPO:PA}
0\in {\mathcal{U}}_0\end{equation}
and~$u\ge0$ in~$\R^n\setminus\Omega$.
Suppose that~$R\in(0,R_o]$ and
\begin{equation}\label{LIBO}
\PPer_\sigma(E,\Omega)+R^{n-\sigma}
\Per_\sigma(B_1,\R^n) \le Q.\end{equation}
Then, there exists~$C>0$, possibly depending on~$R_o$, $n$
and~$\sigma$ such that, for any~$x\in B_{R/2}$,
$$ u(x) \le  C\,\sqrt{L_Q}\, |x|^{1-\frac{\sigma}{2}}.$$
\end{theorem}

We observe that condition~\eqref{LIPASS}
is always satisfied if~$\Phi$ is globally Lipschitz,
but the statement of
Theorem~\ref{GFB} is more general, since it may take
into account a locally Lipschitz~$\Phi$, provided that
the domain is small enough to satisfy~\eqref{LIBO}
(indeed, small domains satisfy this condition for locally
Lipschitz~$\Phi$,
as remarked in the forthcoming Lemma~\ref{mAi86y}).\medskip

We also point out that~\eqref{LIBO} may be equivalently written
\begin{equation}\label{ONW}
\PPer_\sigma(E,\Omega)+\Per_\sigma(B_R,\R^n) \le Q.\end{equation}
One natural way to interpret~\eqref{LIBO} (or~\eqref{ONW})
is that once~$\PPer_\sigma(E,\Omega)$ is strictly less than~$Q$
(i.e. strictly less than the size of the interval in which~$\Phi$
is Lipschitz), then~\eqref{LIBO} (and thus~\eqref{ONW}) holds true
as long as~$R$ is sufficiently small.\medskip

The growth result in Theorem~\ref{GFB} implies, as a byproduct,
an interior H\"older regularity result:

\begin{corollary}\label{COCO}
Let~$Q>0$ and assume that~$\Phi$ is Lipschitz continuous in~$[0,Q]$,
with Lipschitz constant bounded by~$L_Q$.

Assume that~$(u,E)$ is a minimal pair in~$\Omega$,
with~$B_R\Subset\Omega$ and~$u\ge0$ in~$\R^n\setminus\Omega$.

Suppose that~$\PPer_\sigma(E,\Omega)+
R^{n-\sigma} \Per_\sigma(B_1,\R^n) \le Q$ and that~$u\le M$ on~$\partial\Omega$.

Then~$u\in C^{1-\frac{\sigma}{2}}(B_{R/4})$, with
$$ \|u\|_{ C^{1-\frac{\sigma}{2}}(B_{R/4}) }\le
C\,\left(\sqrt{L_Q}+\frac{M}{R^{1-\frac{\sigma}{2}}}\right),$$
for some~$C>0$, possibly depending on~$n$
and~$\sigma$.
\end{corollary}

When $\Phi$ is linear, the result in Corollary~\ref{COCO}
was obtained in Theorem~3.1 of~\cite{salsa} if~$\sigma=1$
and in Theorem~1.1 of~\cite{CSV} if~$\sigma\in(0,1)$.
Differently than in our framework, in~\cites{salsa, CSV}
scaling arguments are available, since scaling is compatible
with the minimization procedure.\medskip

Now we investigate the structure of the free boundary
points in terms of local densities of the phases.
Indeed, we show that
the free boundary points always have uniform density from outside~$E$,
according to the following result:

\begin{theorem}[Density estimate from the null side]\label{DENS}
Assume that~$(u,E)$ is a minimal pair in~$\Omega$,
with~$B_R\subseteq\Omega$, $0\in\partial E$ 
and~$u\ge0$ in~$\R^n\setminus\Omega$.
Set
\begin{equation}\label{12sKJ:9yuhh:P}
P=P(E,\Omega,R):=
\PPer_\sigma (E,\Omega) + R^{n-\sigma} \Per_\sigma (B_{1},\R^n)\end{equation}
and assume that
\begin{equation}\label{12sKJ:9yuhh}
{\mbox{$\Phi$ is strictly increasing in the interval $(0,P)$.}}
\end{equation}
Then there exists~$\delta>0$, possibly depending on~$n$
and~$\sigma$ such that, for any~$r\in (0,R/2)$,
$$ |B_r \setminus E|\ge\delta r^n.$$
\end{theorem}

We point out that condition~\eqref{12sKJ:9yuhh}
is always satisfied if~$\Phi$ is strictly increasing in the
whole of~$[0,+\infty)$, but Theorem~\ref{DENS}
is also general enough to take into consideration
the case in which~$\Phi$ is strictly increasing
only in a subinterval, provided that the energy
domain is sufficiently small
to make the perimeter values to lie in the 
strict monotonicity interval of~$\Phi$
(as a matter of fact, the perimeter contributions in small 
domains is small, as we will point out in the forthcoming Lemma~\ref{mAi86y}).\medskip

The investigation of the density properties
of the free boundary is also completed by
the following counterpart of Theorem~\ref{DENS},
which proves the positive density of the set~$E$:

\begin{theorem}[Density estimate from the positive side]\label{DENS2}
Let~$Q>0$ and assume that
\begin{equation}\label{JH:2345aua123}
{\mbox{$\Phi$ is Lipschitz continuous in~$[0,Q]$,
with Lipschitz constant bounded by~$L_Q$.}}
\end{equation}
and that
\begin{equation}\label{an2383:li:12}
{\mbox{$\Phi'\ge c_o$ a.e. in~$[0,Q]$,}}
\end{equation}
for some~$c_o>0$.

Assume that~$(u,E)$ is a minimal pair in~$\Omega$,
with~$B_R\Subset\Omega$,
$0\in\partial E$
and~$u\ge0$ in~$\R^n\setminus\Omega$.
Suppose that
\begin{equation}\label{LIBO:LALL}
\PPer_\sigma(E,\Omega)+R^{n-\sigma}
\Per_\sigma(B_1,\R^n) \le Q.\end{equation}
Then there exists~$\delta_*>0$, possibly depending on~$n$,
$\sigma$, $c_o$ and~$L_Q$, such that, for any~$r\in (0,R/2)$,
$$ |B_r \cap E|\ge\delta_*\, r^n.$$
More explicitly, such~$\delta_*$ can be taken to be of the form
\begin{equation}\label{QaQ45dfg678}
\delta_*:= \delta_o\,\min\left\{ 1,\,\left(\frac{c_o}{L_Q}\right)^{\frac{n}{\sigma}}\right\},\end{equation}
for some~$\delta_o>0$, possibly depending on~$n$
and~$\sigma$.
\end{theorem}

We remark that the results obtained in this paper are new
even in the local case in which~$\sigma=1$. 
Also, we think it is an interesting point of this
paper that all the cases~$\sigma\in(0,1)$ and~$\sigma=1$
are treated simultaneously in a unified fashion.
The methods presented are also general enough to treat
the case~$\sigma=0$ which would correspond to a volume term
(see e.g.~\cites{Maz, DFPV}).
This case is in fact richer of results
and so we will discuss it in detail in a forthcoming paper.\medskip

The rest
of the paper is organized as follows. In Section~\ref{sec:pre} 
we show some preliminary properties of the minimal pair, 
such as existence, harmonicity and subarmonicity properties, 
and comparison principle. We also prove a ``locality'' property for the (either classical or
fractional) perimeter 
and provide a uniform bound on the (classical or
fractional) perimeter of the set in the minimal pair. 

Section~\ref{sec:ex} is devoted to the construction of the counterexample 
in Theorem~\ref{EXAMPLE}. In Section~\ref{sec:equa} we provide the free boundary equation 
and prove Theorem~\ref{FBW}. 

Then we deal with the regularity of the function~$u$ in the minimal pair 
in the one-phase case, and we prove Theorem~\ref{GFB} and Corollary~\ref{COCO} 
in Sections~\ref{sec:growth} and~\ref{sec:coro}, respectively. 
Finally, Sections~\ref{sec:dens} and~\ref{sec:dens2} 
are devoted to the proofs of the density estimates from both sides 
provided by Theorems~\ref{DENS} and~\ref{DENS2}, respectively. 

\medskip 

Since we hope that the paper may be of interest for different
communities (such as scientists working in
free boundary problems, variational methods, partial
differential equations, geometric measure theory
and fractional problems), we made an effort to
give the details of the arguments involved in the
proofs in a clear and widely accessible way.

\section{Preliminaries}\label{sec:pre}

We start with a useful observation about the positivity
sets of sequences of admissible pairs:

\begin{lemma} \label{ADM:L}
Let~$(u_j,E_j)$ be a sequence of admissible pairs.
Assume that~$u_j\to u$ a.e. in~$\R^n$ and~$\chi_{E_j}\to
\chi_E$ a.e. in~$\R^n$, for some~$u$ and~$E$. Then
$u\ge0$ a.e. in~$E$ and~$u\le0$ a.e. in~$E^c$.\end{lemma}

\begin{proof} We show that $u\ge0$ a.e. in~$E$ (the other claim
being analogous). For this, we write~$\R^n=X\cup Z$,
with~$|Z|=0$ and such that for any~$x\in X$ we have that
$$ \lim_{j\to+\infty} u_j(x)=u(x) \;{\mbox{ and }}\;
\lim_{j\to+\infty}\chi_{E_j}(x) =\chi_E(x).$$
Let now~$x\in E\cap X$. Then
$$ \lim_{j\to+\infty}\chi_{E_j}(x) =\chi_E(x)=1$$
and so there exists~$j_x\in\N$ such that~$\chi_{E_j}(x)\ge 1/2$
for any~$j\ge j_x$. Since the image of a characteristic 
function lies in~$\{0,1\}$, this implies that~$\chi_{E_j}(x)=1$
for any~$j\ge j_x$, and therefore~$u_j(x)\ge0$ for any~$j\ge j_x$.
Taking the limit, we obtain that~$u(x)\ge0$.
Since this is valid for any~$x\in E\cap X$
and~$E\cap X^c\subseteq Z$, which has null measure, we have obtained
the desired result.
\end{proof}

Now we recall a useful auxiliary identity for the (classical or fractional)
perimeter:

\begin{lemma}[``Clean cut'' Lemma]\label{HJ:LEM1}
Let~$\Omega'\Subset\Omega$.
Assume that~$\Per_\sigma(E,\Omega) <+\infty$
and~$\Per_\sigma(F,\Omega)<+\infty$.
Suppose also that
\begin{equation}\label{ssJK:1qs}
E\setminus\overline{\Omega'}=F\setminus\overline{\Omega'}.\end{equation} Then
\begin{equation}\label{ssJK:1qs:BIS}
\Per_\sigma(E,\Omega) -\Per_\sigma(F,\Omega) =
\Per_\sigma(E,\overline{\Omega'})-\Per_\sigma(F,\overline{\Omega'}).
\end{equation}
If in addition~$\PPer_\sigma(E,\Omega) <+\infty$
and~$\PPer_\sigma(F,\Omega)<+\infty$, then
\begin{equation}\label{ssJK:1qs:BIS:2}
\PPer_\sigma(E,\Omega) -\PPer_\sigma(F,\Omega) =
\Per_\sigma(E,\overline{\Omega'})-\Per_\sigma(F,\overline{\Omega'}).
\end{equation}
\end{lemma}

\begin{proof}
For completeness, we distinguish the cases~$\sigma=1$ and~$\sigma\in(0,1)$.
If~$\sigma=1$, we write the perimeter of~$E$ in term of the
Gauss-Green measure~$\mu_E$ (see Remark 12.2 in~\cite{maggi}), namely
$$ \Per (E,\Omega)=|\mu_E| (\Omega).$$
So we define
\begin{equation}\label{U DEF L}
U:=\Omega\setminus\overline{\Omega'}.\end{equation}
We remark that~$U$ is open and~$\Omega=\overline{\Omega'}\cup U$,
with disjoint union. Thus we obtain
\begin{equation}\label{An53ffg1a1}
\begin{split}
& \Per(E,\Omega) -\Per(F,\Omega) -
\Per(E,\overline{\Omega'})+\Per(F,\overline{\Omega'})
\\
=\;& |\mu_E|(\Omega)-|\mu_F|(\Omega)
-|\mu_E|(\overline{\Omega'}) +|\mu_F|(\overline{\Omega'})\\
=\;& |\mu_E|(\overline{\Omega'}\cup U)-|\mu_F|(\overline{\Omega'}\cup U)
-|\mu_E|(\overline{\Omega'}) +|\mu_F|(\overline{\Omega'}) \\
=\;&
|\mu_E|(\overline{\Omega'})+|\mu_E|(U)-|\mu_F|(\overline{\Omega'})-|\mu_F|(U)
-|\mu_E|(\overline{\Omega'}) +|\mu_F|(\overline{\Omega'}) \\
=\;&
|\mu_E|(U)-|\mu_F|(U)\\
=\;& \Per(E,U) -\Per(F,U).
\end{split}\end{equation}
Now we observe that
$$ E\cap U = E\cap \big( \Omega\setminus\overline{\Omega'} \big)
= E\cap \Omega\cap (\overline{\Omega'})^c
= (E\setminus \overline{\Omega'})\cap\Omega,$$
and a similar set identity holds for~$F$.
Thus, by~\eqref{ssJK:1qs}, it follows that~$E\cap U=F\cap U$.
Therefore, by the locality of the classical perimeter
(see e.g. Proposition~3.38(c) in~\cite{AFP}), we obtain
$$ \Per(E,U)=\Per(F,U).$$
If one inserts this into~\eqref{An53ffg1a1}, then obtains~\eqref{ssJK:1qs:BIS}
when~$\sigma=1$.

Now we deal with the case~$\sigma\in(0,1)$.
For this we use~\eqref{9g12:LA} and~\eqref{U DEF L}
and we get that
\begin{eqnarray*}
&& \Per_\sigma (E,\Omega) - \Per_\sigma (E,\overline{\Omega'})\\
&=& L(E\cap \Omega,E^c) + L(E^c\cap\Omega,E\setminus\Omega)
- L(E\cap \overline{\Omega'},E^c) -
L(E^c\cap\overline{\Omega'},E\setminus\overline{\Omega'}) \\
&=& L(E\cap \overline{\Omega'},E^c) +L(E\cap U,E^c) 
+ L(E^c\cap\overline{\Omega'},E\setminus\Omega)
+ L(E^c\cap U,E\setminus\Omega)\\&&\qquad
- L(E\cap \overline{\Omega'},E^c) 
-L(E^c\cap\overline{\Omega'},E\setminus\Omega)
-L(E^c\cap\overline{\Omega'},E\cap U)
\\
&=& 
L(E\cap U,E^c)
+ L(E^c\cap U,E\setminus\Omega)
-L(E^c\cap\overline{\Omega'},E\cap U)
\\
&=& L(E\cap U, E^c\setminus \overline{\Omega'})
+ L(E^c\cap U,E\setminus\Omega),\end{eqnarray*}
and a similar formula holds for~$F$ replacing~$E$.
Now, from~\eqref{ssJK:1qs}, we see that
$$ E\cap U=F\cap U,\qquad
E^c\cap U =F^c\cap U,\qquad
E^c\setminus \overline{\Omega'}=
F^c\setminus \overline{\Omega'}\qquad{\mbox{and}}\qquad
E\setminus\Omega=F\setminus\Omega,$$ thus we obtain~\eqref{ssJK:1qs:BIS}
when~$\sigma\in(0,1)$.

Now, to prove~\eqref{ssJK:1qs:BIS:2}, we can focus on the case~$\sigma=1$
(since~$\PPer_\sigma =\Per_\sigma$ when~$\sigma\in(0,1)$,
thus in this case we return simply to~\eqref{ssJK:1qs:BIS}).
To this end, we observe that~$\Omega'\Subset \Omega_\Upsilon$
(recall formula~\eqref{PPP:D}),
so we can apply~\eqref{ssJK:1qs:BIS} to the sets~$\Omega'$ and~$\Omega_\Upsilon$
and obtain, when~$\sigma=1$,
$$
\PPer_\sigma(E,\Omega) -\PPer_\sigma(F,\Omega) =
\Per(E,\Omega_\Upsilon) -\Per(F,\Omega_\Upsilon)
= \Per(E,\overline{\Omega'}) -\Per(F,\overline{\Omega'}).$$
This completes the proof of~\eqref{ssJK:1qs:BIS:2}.
\end{proof}

Now we state the basic existence result
for the minimizers of the functional in~\eqref{FUNCT}:

\begin{lemma}[Existence of minimal pairs]\label{EX}
Fixed an admissible pair~$(u_o,E_o)$ such 
that~${\mathcal{E}}_\Omega(u_o,E_o)<+\infty$, there exists a minimal
pair~$(u,E)$ in~$\Omega$ such that~$u-u_o\in H^1_0(\Omega)$ 
and~$E\setminus\Omega$ coincides with~$E_o\setminus\Omega$ up to sets of 
measure zero.
\end{lemma}

\begin{proof} 
Let~$(u_j,E_j)$ be a minimizing sequence, namely
\begin{equation}\label{F6ty}
\lim_{j\to+\infty} {\mathcal{E}}_\Omega(u_j,E_j) =
\inf_{X_\Omega(u_o,E_o)} {\mathcal{E}}_\Omega,
\end{equation}
where~$X_\Omega(u_o,E_o)$ denotes the family of all admissible pairs~$(v,F)$ in~$\Omega$ such that~$v-u_o\in H^1_0(\Omega)$
and~$F\setminus\Omega$ coincides with~$E_o\setminus\Omega$ up to sets of
measure zero.

We stress that
$$ \sup_{j\in\N} \Phi\Big( \PPer_\sigma(E_j,\Omega)\Big) <+\infty,$$
thanks to~\eqref{F6ty}. By this and~\eqref{COE},
we obtain that
$$ \sup_{j\in\N} \Per_\sigma(E_j,\Omega)<+\infty.$$
Using this and~\eqref{F6ty},
by compactness (see e.g. Corollary 3.49 in~\cite{AFP} for the case~$\sigma=1$
or Theorem~7.1 in~\cite{guida} for the case~$\sigma\in(0,1)$),
we obtain that, up to subsequences, $u_j$ converges to some~$u$ weakly 
in~$H^1(\Omega)$ and strongly in~$L^2(\Omega)$, and~$\chi_{E_j}$ 
converges to some~$\chi_E$ strongly in~$L^1(\Omega)$, as~$j\to+\infty$.
By Lemma~\ref{ADM:L}, we have that~$(u,E)$ is an admissible pair,
and so by construction
\begin{equation}\label{KJ:8990}
(u,E)\in X_\Omega(u_o,E_o)
.\end{equation}
Also, by the lower semicontinuity (or Fatou Lemma, see e.g.
Proposition~3.38(b) in~\cite{AFP} for the case~$\sigma=1$) we have that
\begin{eqnarray*}
&& \liminf_{j\to+\infty} \int_\Omega |\nabla u_j(x)|^2\,dx \ge
\int_\Omega |\nabla u(x)|^2\,dx \\
{\mbox{and }}
&& \liminf_{j\to+\infty} \PPer_\sigma (E_j,\Omega)\ge
\PPer_\sigma (E,\Omega)
\end{eqnarray*}
and so, using also the monotonicity and the lower semicontinuity of~$\Phi$,
$$ 
\liminf_{j\to+\infty} \Phi\Big(\PPer_\sigma (E_j,\Omega)\Big)\ge
\Phi\Big( \liminf_{j\to+\infty}\PPer_\sigma (E_j,\Omega)\Big)\ge \Phi\Big(
\PPer_\sigma (E,\Omega)\Big).$$
These inequalities and~\eqref{F6ty} give that
$$ {\mathcal{E}}_\Omega(u,E)\le
\inf_{X_\Omega(u_o,E_o)} {\mathcal{E}}_\Omega,$$
and then equality holds in the formula above, thanks
to~\eqref{KJ:8990}.
\end{proof}

As it often happens in free boundary problems (see e.g.~\cites{alt, salsa, CSV}),
the solutions are harmonic in the positivity or negativity sets. 
This happens also in our case, as clarified by the following observation:

\begin{lemma}\label{LE:AR}
Let~$(u,E)$ be a minimal pair in~$\Omega$. Let~$U$ be an open set.
Assume that
either~$\displaystyle\inf_U u>0$
or~$\displaystyle\sup_U u<0$. Then~$u$ is harmonic
in~$U$.
\end{lemma}

\begin{proof} The proof is standard, but we give the
details for the facility of the reader.
We suppose that
\begin{equation}\label{K12s67}
\inf_U u>0, \end{equation}
the other case being
similar.
Let~$x_o\in U$. Since~$U$ is open, there exists~$r>0$
such that~$B_{r}(x_o)\subset U$. Let~$\psi\in C^\infty_0 (B_{r/2}(x_o))$.
Let also~$u_\epsilon:=u+\epsilon\psi$ and
$$ m:= \inf_{\overline{ B_{r/2}(x_o) }} u.$$
By \eqref{K12s67}, we know that~$m>0$. Thus, if~$\epsilon\in\R$,
with~$|\epsilon|<(1+\|\psi\|_{L^\infty(\R^n)})^{-1}m$, 
we have that~$u_\epsilon
\ge u-\epsilon \|\psi\|_{L^\infty(\R^n)}\ge0$ in~$B_{r/2}(x_o)$.
This and the fact that~$\psi$ vanishes outside~$B_{r/2}(x_o)$
give that~$(u_\epsilon,E)$ is an admissible pair.
Thus, the minimality of~$(u,E)$ gives that
$$ 0\le {\mathcal{E}}_{\Omega}(u_\epsilon,E)
-{\mathcal{E}}_{\Omega}(u,E) =\int_\Omega \Big(
|\nabla u(x)+\epsilon\nabla\psi(x)|^2
-
|\nabla u(x)|^2\Big)\,dx,$$
from which the desired result easily follows.
\end{proof}

As it often happens in free boundary problems,
the minimizers satisfy the following subharmonicity property:

\begin{lemma}\label{HJ:AR678a}
Let~$(u,E)$ be a minimal pair in $\Omega$ and~$u^+:=\max\{u,0\}$
and~$u^- := u^+ - u = -\min\{u,0\}$. Then both~$u^+$ and~$u^-$
are subharmonic in~$\Omega$, in the sense that
$$ \int_\Omega \nabla u^\pm (x)\cdot \nabla \psi(x)\,dx\le0,$$
for any~$\psi \in H^1_0(\Omega)$, with~$\psi\ge0$ a.e. in~$\Omega$.
\end{lemma}

\begin{proof} The proof is a modification of the one in Lemma~2.7
in~\cite{salsa}, where this result was proved for the case in which~$\Phi$
is the identity and~$\sigma=1$. We give the details for the facility
of the reader. We argue for~$u^+$, since a similar reasoning works
for~$u^-$. We define~$v^\star$ to be the harmonic replacement of~$u^+$
in~$\Omega$ which vanishes in~$E^c$, that is the minimizer of the Dirichlet
energy in~$\Omega$ among all the functions~$v$ in~$H^1(\Omega)$ such that~$
v-u^+\in H^1_0(\Omega)$ and~$v=0$ a.e. in~$E^c$.
For the existence and the uniqueness of the harmonic replacement see e.g.
Section~2 in~\cite{salsa} or Lemma~2.1 in~\cite{HAR}.
In particular, the uniqueness result gives that
\begin{equation}\label{EF6U9:1}
\begin{split}
&{\mbox{if~$v$ in~$H^1(\Omega)$ is such that~$
v-u^+\in H^1_0(\Omega)$, $v=0$ a.e. in~$E^c$}}\\
&{\mbox{and }}\,
\int_\Omega |\nabla v(x)|^2\,dx\le
\int_\Omega |\nabla v^\star(x)|^2\,dx,\,
{\mbox{then $v=v^\star$ a.e. in~$\R^n$.}}
\end{split}
\end{equation}
Moreover, by Lemma~2.3 in~\cite{salsa},
we have that
\begin{equation}\label{EF6U9:2}
{\mbox{$v^\star$ is subharmonic.}}
\end{equation}
We also notice that~$v^\star\ge0$ by the classical maximum
principle and therefore~$(v^\star,E)$
is an admissible pair. Then, the minimality of~$(u,E)$
implies that
\begin{eqnarray*}
0 &\ge&
{\mathcal{E}}_{\Omega}(u,E) 
-
{\mathcal{E}}_{\Omega}(v^\star,E)\\
&=&
\int_\Omega |\nabla u(x)|^2\,dx-
\int_\Omega |\nabla v^\star (x)|^2\,dx\\
&\ge&
\int_\Omega |\nabla u^+(x)|^2\,dx-  
\int_\Omega |\nabla v^\star (x)|^2\,dx.\end{eqnarray*}
This implies that~$u^+$ coincides with~$v^\star$, thanks to~\eqref{EF6U9:1},
and so it is subharmonic, in light of~\eqref{EF6U9:2}.
\end{proof}

\begin{rem}\label{RE:su} In light of Lemma~\ref{HJ:AR678a},
we have (see e.g. Proposition 2.2 in~\cite{giaquinta}) that
the map
$$ R\to \frac{1}{|B_R|} \int_{B_R(p)} u^+(x)\,dx$$
is monotone nondecreasing, therefore, up to changing~$u^+$
in a set of measure zero, we can (and implicitly do from now on)
suppose that
$$ u(p)=\lim_{\epsilon\searrow0}
\frac{1}{|B_\epsilon|} \int_{B_\epsilon (p)} u^+(x)\,dx.$$
\end{rem}

Another simple and interesting property of
the solution is given by the following maximum principle:

\begin{lemma}\label{MAX:PLE}
Assume that
\begin{equation}\label{6tHHujKaK}
{\mbox{$\Phi(0)<\Phi(t)$ for any~$t>0$.}}\end{equation}
Let $(u, E)$ be a minimal pair in~$\Omega$ and let $a\in\R$. 
If~$u\le a$ in~$\Omega^c$, then~$u\le a$ in the whole of~$\R^n$.

Similarly, if~$u\ge a$ in~$\Omega^c$, then~$u\ge a$ in the whole of~$\R^n$.
\end{lemma}

\begin{proof} We suppose that
\begin{equation}\label{FG6722}
{\mbox{$u\ge a$ in~$\Omega^c$,}}
\end{equation}
the other case being analogous.

We need to distinguish the cases~$a\le0$ and~$a>0$.

If~$a\le 0$, we take~$u^\star :=\max\{u,a\}$.
Notice that~$(u^\star,E)$ is an admissible pair:
indeed, a.e. in~$E$ we have that~$0\le u\le u^\star$,
while a.e. in~$E^c$ we have that~$u\le0$ and so~$u^\star\le0$.
Also, by~\eqref{FG6722}, we have that~$u\ge a$
in~$\Omega^c$, and so~$u^\star = u$ in~$\Omega^c$.
As a consequence, the minimality of~$(u,E)$ gives that
$$ 0\le {\mathcal{E}}_{\Omega}(u^\star,E) -{\mathcal{E}}_{\Omega}(u,E)
= \int_\Omega \Big(|\nabla u^\star(x)|^2 -|\nabla u(x)|^2\Big)\,dx
=-\int_{\Omega\cap \{ u<a\}}|\nabla u(x)|^2\,dx,$$
which implies that~$u\ge a$, as desired.

Now suppose that~$a>0$.
We take~$u^\sharp$ to be the minimizer of the Dirichlet
energy in~$\Omega$ with trace datum~$u$ along~$\partial\Omega$
(and thus we set~$u^\sharp:=u$ outside~$\Omega$); then we have that
\begin{equation}\label{6c3ty}
\Gamma:=\int_\Omega |\nabla u (x)|^2\,dx
-
\int_\Omega |\nabla u^\sharp(x)|^2\,dx\ge 0.
\end{equation}
Moreover, by~\eqref{FG6722} and the classical maximum principle,
we know that
\begin{equation}\label{8uf556HH}
{\mbox{$u^\sharp\ge a$ in the whole of~$\R^n$.}}\end{equation}
Thus, $u^\sharp>0$ and so~$(u^\sharp,\R^n)$ is an admissible pair.
Accordingly, the minimality of~$(u,E)$ and~\eqref{6c3ty}
give that
\begin{equation}\label{74d99q}
\begin{split}
0 \;&\le {\mathcal{E}}_{\Omega}(u^\sharp,\R^n) -{\mathcal{E}}_{\Omega}(u,E)
\\&= \int_\Omega |\nabla u^\sharp(x)|^2\,dx
+\Phi(0)-\int_\Omega |\nabla u (x)|^2\,dx -\Phi\Big( \PPer_\sigma(E,\Omega)\Big)
\\ &= -\Gamma+ 
\Phi(0)-\Phi\Big( \PPer_\sigma(E,\Omega)\Big).\end{split}\end{equation}
As a consequence, 
$$ \Phi\Big( \PPer_\sigma(E,\Omega)\Big)\le -\Gamma+\Phi(0)
\le \Phi(0),$$
hence, exploiting~\eqref{6tHHujKaK},
we see that~$\PPer_\sigma(E,\Omega)=0$.
Plugging this information into~\eqref{74d99q}, we obtain
that~$ 0\le -\Gamma$ and thus, recalling~\eqref{6c3ty},
we conclude that~$\Gamma=0$.
By the uniqueness of the minimizer of
the Dirichlet energy, this implies that~$u^\sharp$ coincides
with~$u$. In light of this and of~\eqref{8uf556HH}, we have that~$u=u^\sharp
\ge a$, as desired.
\end{proof}

Now we give a uniform bound on the (classical or fractional) perimeter
of the sets in the minimal pairs:

\begin{lemma}\label{mAi86y}
Suppose that~$\Omega$ is strictly starshaped (i.e.~$t\overline\Omega
\subseteq\Omega$ for any~$t\in(0,1)$) and that
\begin{equation}\label{ZXJH:INA}
{\mbox{$\Phi$ is strictly monotone.}}
\end{equation}
Let~$(u,E)$ be a minimal pair in~$\Omega$. Assume that~$u\ge0$.
Then, for any~$\Omega'\subseteq \Omega$,
with~$\Omega'$ open, Lipschitz and bounded, we have that
\begin{equation}\label{7y45:aa1}
\Per_\sigma (E,\Omega')\le
2\Per_\sigma (\Omega',\R^n).\end{equation}
In particular, if~$\Omega\supseteq B_R$,
then, for any~$r\in(0,R]$,
\begin{equation}\label{7y45:aa2}
\Per_\sigma (E,B_r)\le Cr^{n-\sigma},\end{equation}
for some~$C>0$ possibly depending on~$n$ and~$\sigma$.
\end{lemma}

\begin{proof} We observe that~\eqref{7y45:aa2}
follows from~\eqref{7y45:aa1} by taking~$\Omega':=B_r$,
so we focus on the proof of~\eqref{7y45:aa1}.
For this, first we suppose that~$\Omega'\Subset\Omega$
(the general case in which~$\Omega'\subseteq\Omega$
will be considered at the end of
the proof, by a limit procedure). Let~$F:=E\cup \Omega'$.
Notice that~$F\setminus \overline{\Omega'} =E\cup \Omega'\cap (\overline{\Omega'})^c
= E\setminus \overline{\Omega'}$. Thus, by formula~\eqref{ssJK:1qs:BIS:2}
in Lemma~\ref{HJ:LEM1},
we get that
\begin{equation}\label{8gf3erfh}
\PPer_\sigma(E,\Omega) -\PPer_\sigma(F,\Omega) =
\Per_\sigma(E,\overline{\Omega'})-\Per_\sigma(F,\overline{\Omega'})
.\end{equation}
Now, let~$v$ be the minimizer of the Dirichlet energy
in~$\Omega'$ with trace datum~$u$ along $\partial \Omega'$
(then take~$v:=u$ outside~$\Omega'$). Since~$u\ge0$, then so is~$v$.
Hence, the pair~$(v,F)$ is admissible. Therefore,
the minimality of~$(u,E)$
implies that
\begin{equation*}
\begin{split}
0\; &\le
{\mathcal{E}}_\Omega (v,F) -
{\mathcal{E}}_\Omega (u,E)
\\ &= \int_{\Omega'} |\nabla v(x)|^2\,dx
- \int_{\Omega'} |\nabla u(x)|^2\,dx
+\Phi\Big(\PPer_\sigma (F,\Omega)\Big) -\Phi\Big(\PPer_\sigma (E,\Omega)\Big)
\\ &\le 0+
\Phi\Big(\PPer_\sigma(F,\Omega)\Big)
-\Phi\Big(\PPer_\sigma(E,\Omega)\Big)
.\end{split}\end{equation*}
Hence, by~\eqref{ZXJH:INA}, we have that~$\PPer_\sigma(E,\Omega)\le
\PPer_\sigma(F,\Omega)$ and so, by~\eqref{8gf3erfh},
\begin{equation}\label{X12}
\Per_\sigma(E,\overline{\Omega'})-\Per_\sigma(F,\overline{\Omega'})
=
\PPer_\sigma(E,\Omega) -\PPer_\sigma(F,\Omega) \le0.
\end{equation}
In addition, we have that
$$ \Per_\sigma(F,\overline{\Omega'})
=\Per_\sigma(E\cup \Omega',\overline{\Omega'})
\le 2\Per_\sigma(\Omega',\R^n),$$
where the last formula follows using~\eqref{9g12:LA}
if~$\sigma\in(0,1)$ and, for instance, formula~(16.12)
in~\cite{maggi} when~$\sigma=1$.

The latter inequality and~\eqref{X12} give that
$$ \Per_\sigma(E,{\Omega'})\le
\Per_\sigma(E,\overline{\Omega'})\le \Per_\sigma(F,\overline{\Omega'})\le
2\Per_\sigma(\Omega',\R^n).$$
This proves the desired result when~$\Omega'\Subset\Omega$.
Let us now deal with the case~$\Omega'\subseteq\Omega$.
For this, we set~$\Omega_\epsilon':=(1-\epsilon)\Omega'$.
Since~$\Omega$ is 
strictly starshaped, we have that~$\overline{\Omega_\epsilon'}=
(1-\epsilon)\overline{\Omega'}\subseteq (1-\epsilon)\overline{\Omega}\subseteq
\Omega$ for any~$\epsilon\in(0,1)$, so we can use the result
already proved and we get that
\begin{equation}\label{7y45:aa1:epsilon}
\Per_\sigma (E,\Omega'_\epsilon)\le
2\Per_\sigma (\Omega'_\epsilon,\R^n).\end{equation}
Moreover,
\begin{equation}\label{7y45:aa1:epsilon:2}
\Per_\sigma (\Omega'_\epsilon,\R^n) = (1-\epsilon)^{n-\sigma}
\Per_\sigma (\Omega',\R^n).\end{equation}
Also, we claim that
\begin{equation}\label{CLya}
\lim_{\epsilon\searrow0} \Per_\sigma(E,{\Omega_\epsilon'})=\Per_\sigma(E,{\Omega'}).\end{equation}
To prove it, we distinguish the cases~$\sigma=1$ and~$\sigma\in(0,1)$.
If~$\sigma=1$, we use the representation
of the perimeter of~$E$ in term of the
Gauss-Green measure~$\mu_E$ (see Remark 12.2 in~\cite{maggi})
and the Monotone Convergence Theorem (applied
to the monotone sequence of sets~${\Omega_\epsilon'}$, see e.g.
Theorem~1.26(a) in~\cite{YEH}): in this way, we have
$$ \lim_{\epsilon\searrow0} \Per(E,\Omega_\epsilon')=
\lim_{\epsilon\searrow0}|\mu_E|(\Omega_\epsilon')=
|\mu_E|(\Omega')=
\Per(E,\Omega').$$
This proves~\eqref{CLya}
when~$\sigma=1$. If instead~$\sigma\in(0,1)$,
we first observe that~$\Per_\sigma(E,\Omega_\epsilon')\le \Per_\sigma(E,\Omega')$
and then
\begin{equation}\label{CLya-2}
\limsup_{\epsilon\searrow0} \Per_\sigma(E,{\Omega_\epsilon'})\le \Per_\sigma(E,{\Omega'}).\end{equation}
Conversely, we use~\eqref{9g12:LA} to write
\begin{eqnarray*}
\Per_\sigma (E,\Omega_\epsilon') &=&
L(E\cap \Omega_\epsilon',E^c) + L(E^c\cap \Omega_\epsilon',E\cap (\Omega_\epsilon')^c)\\
&\ge& L(E\cap \Omega_\epsilon',E^c) + L(E^c\cap \Omega_\epsilon',E\cap (\Omega')^c).\end{eqnarray*}
Consequently, by taking the limit here above
and using Fatou's Lemma,
$$ \liminf_{\epsilon\searrow0} \Per_\sigma(E,{\Omega_\epsilon'})\ge
L(E\cap \Omega',E^c) + L(E^c\cap \Omega',E\cap (\Omega')^c)
=\Per_\sigma(E,\Omega').$$
This, together with~\eqref{CLya-2}, establishes~\eqref{CLya}.

Now, combining~\eqref{7y45:aa1:epsilon}, \eqref{7y45:aa1:epsilon:2}
and~\eqref{CLya}, we obtain~\eqref{7y45:aa1} by taking a limit in~$\epsilon$.
\end{proof}

\section{Proof of Theorem~\ref{EXAMPLE}}\label{sec:ex}

Now we prove Theorem~\ref{EXAMPLE}. The idea of the proof
is that, on the one hand, for large balls, we obtain a large contribution
of the perimeter, which makes the energy functional
simply the Dirichlet energy plus a constant, due to the
special form of~$\Phi$. On the other hand,
for small balls, both the Dirichlet energy and the perimeter
give small contribution, and in this range the
contribution of the perimeter becomes predominant.
This dichotomy of the energy behavior makes the
minimal pair change accordingly, namely, in 
large balls, harmonic functions are favored, somehow
independently of their level sets, while, conversely,
for small balls the sets which minimize the perimeter
are favored, somehow independently on the Dirichlet
energy of the function that they support. That is, in the end,
the core of the counterexample is, roughly speaking, that being a minimal surface
is something rather different than being the level set of
a harmonic function.

Of course, some computations are needed to justify the above
heuristic arguments and we present now all the details of the proof.

\subsection{Estimates on~$\Per_\sigma(E,B_R)$ from below}

Here we obtain bounds from below
for the (either classical or fractional)
perimeter of a set~$E$ in~$B_R$, once~$E$ is ``suitably fixed'' 
outside\footnote{For simplicity, we state and prove all the results
of this part only in~$\R^2$, though some of the arguments would
also be valid in higher dimension.}
the ball~$B_R\subset\R^2$.
For this scope, we recall the notation in~\eqref{tilde u}
and~\eqref{tilde E},
and we have:

\begin{lemma}\label{PER:1}
Let~$c_o>0$.
Let~$(u,E)$ be an admissible pair in~$\R^2$.
Assume that~$u-\tilde u\in H^1_0(B_1)$
and that
\begin{equation*}
\int_{B_1}|\nabla u(X)|^2\,dX \le c_o.
\end{equation*}
Then there exists~$c>0$, possibly depending on~$c_o$,
such that
\begin{equation}\label{JH86:l} \Per_\sigma(E,B_1)\ge c.\end{equation}\end{lemma}

\begin{proof} We argue by contradiction. If the thesis in~\eqref{JH86:l}
were false, there would exist a sequence of
admissible pairs~$(u_j,E_j)$ such that~$u_j-\tilde u\in H^1_0(B_1)$,
\begin{equation*}
\int_{B_1}|\nabla u_j(X)|^2\,dX \le c_o\end{equation*}
and
\begin{equation}\label{JH86:l:BIS}
\Per_\sigma(E_j,B_1)\le \frac{1}{j}.\end{equation}
Thus, by compactness,
(see e.g. Corollary 3.49 in~\cite{AFP} for the case~$\sigma=1$
or Theorem~7.1 in~\cite{guida} for the case~$\sigma\in(0,1)$),
we conclude that, up to subsequences, $u_j$ converges to some~$u_\infty$ weakly
in~$H^1(B_1)$ and strongly in~$L^2(B_1)$, 
with
\begin{equation}\label{12J9:lkjuio99}
u_\infty-\tilde u\in H^1_0(B_1)
,\end{equation}
and~$\chi_{E_j}$
converges to some~$\chi_{E_\infty}$ strongly in~$L^1(B_1)$, as~$j\to+\infty$.
Accordingly, by the lower semicontinuity of the
(either classical or fractional) perimeter
(or Fatou Lemma, see e.g.
Proposition~3.38(b) in~\cite{AFP} for the case~$\sigma=1$) we deduce
from~\eqref{JH86:l:BIS} that
$$ \Per_\sigma(E_\infty,B_1)=0.$$
Hence, from the relative isoperimetric inequality
(see e.g. 
Lemma 2.5 in~\cite{ruffini} when~$\sigma\in(0,1)$ and
formula~(12.46) in~\cite{maggi} when~$\sigma=1$),
$$ \min\left\{ |B_1\cap E_\infty|^{\frac{2-\sigma}{2}},
|B_1\setminus E_\infty|^{\frac{2-\sigma}{2}}\right\}\le \hat C\;
\Per_\sigma(E_\infty,B_1)=0,$$
for some~$\hat C>0$.
Thus, we can suppose that
\begin{equation}\label{1Jh7:L}
|B_1\cap E_\infty|=0,\end{equation}
the case~$|B_1\setminus E_\infty|=0$ being similar.
Also, in virtue of Lemma~\ref{ADM:L}, we have that
$u_\infty\ge0$ a.e. in~$E_\infty$ and~$u_\infty\le0$ a.e. in~$E^c_\infty$.
Thus, by~\eqref{1Jh7:L}, we obtain that~$u_\infty\le0$ a.e. in~$B_1$.
Looking at a neighborhood of~$\partial B_1$
in the first quadrant,
we obtain that this is in contradiction with~\eqref{12J9:lkjuio99},
thus proving the desired result.
\end{proof}

By scaling Lemma~\ref{PER:1}, we obtain:

\begin{lemma}\label{JK:poi:R}
Let~$c_o>0$ and~$R>0$.
Let~$(u,E)$ be an admissible pair
in~$\R^2$. Assume that~$u-\tilde u\in H^1_0(B_R)$
and that
\begin{equation}\label{jhgu8:k}
\int_{B_R}|\nabla u(X)|^2\,dX \leq c_o\,R^4.
\end{equation}
Then there exists~$c>0$, possibly depending on~$c_o$,
such that
$$ \Per_\sigma(E,B_R)\ge cR^{2-\sigma}.$$
\end{lemma}

\begin{proof} We set
$$ {\mbox{$u_*(X):= R^{-2} \,u(RX)$
and~$E_* := \frac{E}{R}:=\{ X/R {\mbox{ s.t. }}X\in E\}$.}}$$
Notice that~$R^{-2} \,\tilde u(RX) = R^{-2} \,(Rx)\,(Ry)=\tilde u(X)$,
therefore~$u_*-\tilde u\in H^1_0(B_1)$. Also, $(u_*,E_*)$
is an admissible pair.
In addition,
$$ \int_{B_1} |\nabla u_*(X)|^2\,dX
= R^{-2} \int_{B_1} |\nabla u(RX)|^2\,dX
= R^{-4} \int_{B_R} |\nabla u(Y)|^2\,dY \le c_o,$$
thanks to~\eqref{jhgu8:k}. As a consequence,
we are in the position of applying
Lemma~\ref{PER:1} to the pair~$(u_*,E_*)$ and thus we obtain that
$$ c\le \Per_\sigma(E_*,B_1) =
\Per_\sigma\left(\frac{E}{R},\frac{B_R}{R}\right)
= \frac{1}{R^{2-\sigma}}
\Per_\sigma(E,B_R),$$
as desired.
\end{proof}

\subsection{Analysis of minimizers in large balls}

Now we give a concrete example of a minimizer in~$B_R\subset \R^2$
for $R$ large enough. To this end,
we consider a monotone nondecreasing and lower semicontinuous
function~$\tilde\Phi:[0,+\infty)\to [0,+\infty)$, with
\begin{equation} \label{ART:tilde}
{\mbox{$\tilde\Phi(t)=1$ for any~$t\in[ 2,+\infty)$. }}\end{equation}
We let
$$ {\tilde{\mathcal{E}}}_\Omega(u,E):=
\int_\Omega |\nabla u(X)|^2\,dX + \tilde\Phi\Big( \PPer_\sigma (E,\Omega)\Big).$$
We remark that, in principle, the minimization
procedure in Lemma~\ref{EX} fails for this functional,
since the coercivity assumption~\eqref{COE}
is not satisfied by~$\tilde\Phi$. Nevertheless,
we will be able to construct explicitly a minimizer
for large balls of~${\tilde{\mathcal{E}}}$.
Then, we will modify~$\tilde\Phi$ at infinity and we
will obtain from it a minimizer for a functional
of the type in~\eqref{FUNCT}, with a coercive~$\Phi$. The details go
as follows.

\begin{proposition}\label{LARGE BALLS:0}
Let~$n=2$.
Let~$\tilde u$
and~$\tilde E$ be as in~\eqref{tilde u} and~\eqref{tilde E}.

Then, there exists~$R_o>0$, only depending on~$n$ and~$\sigma$,
such that if~$R\ge R_o$ then
\begin{equation}\label{h64r5678uytf}
{\tilde{\mathcal{E}}}_{B_R}(\tilde u,\tilde E)\le
{\tilde{\mathcal{E}}}_{B_R}(v,F),\end{equation}
for any admissible pair~$(v,F)$ such that~$v-\tilde u\in H^1_0(B_R)$
and~$F\setminus B_R=\tilde E\setminus B_R$, up to sets of measure zero.
\end{proposition}

\begin{proof} We observe that~$\nabla \tilde u(x,y)=(y,x)$, and so
\begin{equation}\label{PO:K7tf456-1}
\int_{B_R}
|\nabla \tilde u(X)|^2\,dX =
\int_{B_R} |X|^2\,dX \le C_1 R^4,\end{equation}
for some~$C_1>0$.
Moreover, since~$\tilde E$ is a cone, we have that~$\tilde E= R\tilde E$,
thus
$$ \Per_\sigma (\tilde E,B_R) =\Per_\sigma (R \tilde E,\,R B_1)
=C_2 R^{2-\sigma},$$
for some~$C_2>0$. In particular, if~$R\ge (2/C_2)^{\frac{1}{2-\sigma}}$,
we have that
$$ \PPer_\sigma (\tilde E,B_R)\ge
\Per_\sigma (\tilde E,B_R)\ge 2$$ 
and then, by~\eqref{ART:tilde},
\begin{equation}\label{923d61a1}
\tilde\Phi\Big( \PPer_\sigma (\tilde E,B_R) \Big)=1.\end{equation}
This and~\eqref{PO:K7tf456-1} imply that
\begin{equation}\label{L7yTYh:KJH89}
{\tilde{\mathcal{E}}}_{B_R}(\tilde u,\tilde E) \le C_1 R^4 + 1
\le 2C_1 R^4,\end{equation}
if~$R$ is large enough.

Now suppose, by contradiction,
that~\eqref{h64r5678uytf} is violated, i.e.
\begin{equation}\label{8uj111s} {\tilde{\mathcal{E}}}_{B_R}(\tilde u,\tilde E)>
{\tilde{\mathcal{E}}}_{B_R}(v,F),\end{equation}
for some competitor~$(v,F)$.
In particular, by~\eqref{L7yTYh:KJH89},
\begin{equation}\label{76111JK:KH:0} 
\int_{B_R} |\nabla v(X)|^2\,dX \le {\tilde{\mathcal{E}}}_{B_R}(v,F)
\le {\tilde{\mathcal{E}}}_{B_R}(\tilde u,\tilde E)\le 2C_1 R^4.\end{equation}
This says that formula~\eqref{jhgu8:k}
is satisfied by the pair~$(v,F)$
with~$c_o:=2C_1$, and so Lemma~\ref{JK:poi:R}
gives that
$$ \PPer_\sigma(F,B_R)\ge\Per_\sigma(F,B_R)\ge cR^{2-\sigma},$$
for some~$c>0$. In particular, for large~$R$, we have that
$$ \tilde\Phi\Big( \PPer_\sigma(F,B_R) \Big)=1$$
and therefore
\begin{equation}\label{UD5fG6} 
{\tilde{\mathcal{E}}}_{B_R}(v,F)=
\int_{B_R} |\nabla v(X)|^2\,dX +1.\end{equation}
On the other hand, since~$\tilde u$ is harmonic,
$$ \int_{B_R} |\nabla v(X)|^2\,dX\ge \int_{B_R} |\nabla \tilde u(X)|^2\,dX,$$
hence~\eqref{UD5fG6} and~\eqref{923d61a1}
give that
$$ {\tilde{\mathcal{E}}}_{B_R}(v,F)\ge \int_{B_R} |\nabla \tilde u(X)|^2\,dX+1
={\tilde{\mathcal{E}}}_{B_R}(\tilde u,\tilde E).$$
This is in contradiction with~\eqref{8uj111s} and so the desired
result is established.
\end{proof}

\begin{corollary}\label{LARGE BALLS}
Let~$n=2$.
Let~$\tilde u$
and~$\tilde E$ be as in~\eqref{tilde u} and~\eqref{tilde E}.
There exists~$K_o>2$ such that the following statement is true.
Assume that
\begin{equation}\label{8u16ysT7uII}
{\mbox{$\Phi(t)=1$ for any~$t\in[ 2,K_o]$. }}\end{equation}
Then, there exists~$R_o>0$ such that~$(\tilde u,\tilde E)$
is a minimal pair in~$B_{R_o}$.
\end{corollary}

\begin{proof} We define
$$ \tilde\Phi(t):= \left\{
\begin{matrix}
\Phi(t) & {\mbox{ if }} t\in[0,2]\\
1 & {\mbox{ if }} t\in(2,+\infty).
\end{matrix}
\right. $$
Then we are in the setting of Proposition~\ref{LARGE BALLS:0}
and we obtain that there exists~$R_o>0$, 
only depending on~$n$ and~$\sigma$, such that~$(\tilde u,\tilde E)$
is a minimal pair for~${\tilde{\mathcal{E}}}_{B_{R_o}}$.
So we define
$$ K_o := \PPer_\sigma( \tilde E, B_{R_o}) + 3.$$
Notice that~$K_o$ only depends on~$n$ and~$\sigma$,
since so does~$R_o$, and~$\tilde u$ and~$\tilde E$
are fixed.

To complete the proof of the desired claim, we need
to show that~$(\tilde u,\tilde E)$
is a minimal pair for~${\mathcal{E}}_{B_{R_o}}$, as long as~\eqref{8u16ysT7uII}
is satisfied. For this, we remark that, since~$\Phi$ is monotone,
we have that~$\Phi(t)\ge\Phi(2)=1$, for any~$t\ge2$. As a consequence,
we get that~$\Phi(t)\ge \tilde\Phi(t)$ for any~$t\ge0$.
Therefore, if~$(v,F)$ is a competitor for~$(\tilde u,\tilde E)$,
we deduce from~\eqref{h64r5678uytf} that
\begin{equation}\label{87655ta22} 
{\tilde{\mathcal{E}}}_{B_{R_o}}(\tilde u,\tilde E)\le
{\tilde{\mathcal{E}}}_{B_{R_o}}(v,F) \le {\mathcal{E}}_{B_{R_o}}(v,F).\end{equation}
On the other hand,
\begin{equation}\label{53102}
\PPer_\sigma( \tilde E, B_{R_o}) \le K_o .\end{equation}
Moreover, we have that~$\tilde\Phi(t)=1=\Phi(t)$ if~$t\in(2,K_o]$.
Therefore, we get that~$\tilde\Phi=\Phi$ in~$[0,K_o]$
and thus, by~\eqref{53102},
$$ \tilde\Phi\Big( \PPer_\sigma( \tilde E, B_{R_o})\Big)=
\Phi\Big( \PPer_\sigma( \tilde E, B_{R_o})\Big).$$
By plugging this into~\eqref{87655ta22}, we conclude that
$$ {\mathcal{E}}_{B_{R_o}}(\tilde u,\tilde E)
={\tilde{\mathcal{E}}}_{B_{R_o}}(\tilde u,\tilde E)\le {\mathcal{E}}_{B_{R_o}}(v,F),$$
as desired.
\end{proof}

\subsection{Estimates in small balls}

Here, we show that the minimal pair
constructed in Corollary~\ref{LARGE BALLS}
in large balls does not remain minimal in small balls.

\begin{proposition}\label{SMALL BALLS}
Let~$n=2$.
Assume that
\begin{equation}\label{QDV78}
{\mbox{$\Phi(t) = t^\gamma$ for any~$t\in[0,1]$}}
\end{equation}
for some
\begin{equation}\label{gamma4}
\gamma\in\left(0,\;\frac{4}{2-\sigma}\right).\end{equation}
Let~$\tilde u$
and~$\tilde E$ be as in~\eqref{tilde u} and~\eqref{tilde E}.

Then there exists~$r_o>0$ such that if~$r\in(0,r_o]$ then
the pair~$(\tilde u,\tilde E)$ is not minimal in~$B_{r}$.
\end{proposition}

\begin{proof} We suppose, by contradiction,
that~$(\tilde u,\tilde E)$ is minimal in~$B_{r}$,
with~$r$ sufficiently small.

We observe that~$\tilde E$ is not a minimizer of
the perimeter in~$\overline{B_{1/2}}$ (see~\cite{SV}
for the case~$\sigma\in(0,1)$). Therefore there exists a perturbation~$E_\sharp$ of~$\tilde E$
inside~$\overline{B_{1/2}}$ for which
$$ \Per_\sigma(E_\sharp,\overline{B_{1/2}}) \le \Per_\sigma(\tilde E,\overline{B_{1/2}}) - a,$$
for some (small, but fixed)~$a>0$. As a consequence, recalling Lemma~\ref{HJ:LEM1},
\begin{equation}\label{WD6:PO}
\Per_\sigma(E_\sharp,B_{1}) -\Per_\sigma(\tilde E,B_{1})
=
\Per_\sigma(E_\sharp,\overline{B_{1/2}}) - \Per_\sigma(\tilde E,\overline{B_{1/2}})\le - a.\end{equation}
Now we take~$\psi \in C^\infty(\R^2,\,[0,1])$ such that~$\psi(X)=0$
for any~$X\in B_{3/4}$ and~$\psi(X)=1$ for any~$X\in B_{9/10}^c$.
We define
$$ u_\sharp(X)= u_\sharp(x,y):=\tilde u(X)\,\psi(X)
=xy\,\psi(x,y).$$
We claim that
\begin{equation}\label{QW11JH:1}
{\mbox{$u_\sharp\ge0$ a.e. in~$E_\sharp$ and~$u_\sharp\le0$ a.e. in~$E_\sharp^c$.}}
\end{equation}
To check this, we observe that~$u_\sharp=0$ in~$B_{3/4}$,
so it is enough to prove~\eqref{QW11JH:1}
for points outside~$B_{3/4}$. Then, we also remark that~$E_\sharp
\setminus B_{3/4} = \tilde E\setminus B_{3/4}$,
and, as a consequence,
we get that~$\tilde u\ge0$ a.e. in~$E_\sharp \setminus B_{3/4}$ 
and~$\tilde u\le0$ a.e. in~$E_\sharp^c \setminus B_{3/4}$.
Hence, since~$\psi\ge0$, we obtain that~$u_\sharp
\ge0$ a.e. in~$E_\sharp \setminus B_{3/4}$
and~$u_\sharp\le0$ a.e. in~$E_\sharp^c \setminus B_{3/4}$.
These observations complete the proof of~\eqref{QW11JH:1}.

Now we define
$$ u_r (X):= r^2 u_\sharp \left(\frac{X}{r}\right)=
xy\,\psi\left(\frac{X}{r}\right)=\tilde u(X)\,
\psi\left(\frac{X}{r}\right)$$
and
$$ E_r := r E_\sharp.$$
{F}rom~\eqref{QW11JH:1},
we obtain that~$u_r\ge0$ a.e. in~$E_r$ and~$u_r\le0$ a.e. in~$E_r^c$,
and thus~$(u_r,E_r)$ is an admissible pair.

Now we check that the data of~$(u_r,E_r)$ coincide with~$(\tilde u,\tilde E)$
outside~$B_r$. First of all, we have that~$\psi=1$
in~$B_{9/10}^c$, thus, if~$X\in B_{9r/10}^c$ we have that~$u_r(X)=\tilde u(X)$.
This shows that
\begin{equation}\label{AS5st}
u_r-\tilde u\in H^1_0(B_r).\end{equation}
Moreover,
\begin{eqnarray*}
&& E_r \setminus B_r =
\{ X\in B_r^c {\mbox{ s.t. }}
r^{-1} X\in E_\sharp \} \\
&&\qquad=
\{ X=rY {\mbox{ s.t. }}
Y \in E_\sharp\setminus B_1 \}
=
\{ X=rY {\mbox{ s.t. }}
Y \in \tilde E\setminus B_1 \}.
\end{eqnarray*}
Now, since~$\tilde E$ is a cone, we have that
$Y \in \tilde E$ if and only if~$rY \in \tilde E$,
and so, as a consequence,
$$ E_r \setminus B_r =
\{ X=rY\in \tilde E {\mbox{ s.t. }}
Y \in B_1^c \} = \tilde E\setminus B_r.$$
Using this and~\eqref{AS5st}, we obtain that,
if~$(\tilde u,\tilde E)$ is minimal in~$B_r$, then
\begin{equation}\label{D4600:0}
{\mathcal{E}}_{B_r} (\tilde u,\tilde E)\le
{\mathcal{E}}_{B_r} (u_r,E_r).\end{equation}
Now we remark that, since~$\tilde E$ is a cone,
\begin{equation} \label{612:0011}
\Per_\sigma (\tilde E,B_r) =\Per_\sigma (r\tilde E,\,r B_1)
=r^{2-\sigma} \,\Per_\sigma (\tilde E,\,B_1).\end{equation}
Now we define
$$ \vartheta:=\left\{
\begin{matrix}
4\Upsilon & {\mbox{ if }} \sigma=1,\\
0 & {\mbox{ if }} \sigma\in(0,1),
\end{matrix}
\right.$$
and we claim that
\begin{equation} \label{612:0012}
\PPer_\sigma (\tilde E,B_r) 
=r^{2-\sigma} \,\Per_\sigma (\tilde E,\,B_1)+\vartheta.\end{equation}
Indeed, if~$\sigma\in(0,1)$, then~\eqref{612:0012}
boils down to~\eqref{612:0011}. If instead~$\sigma=1$,
we use~\eqref{612:0011} in the following computation:
\begin{eqnarray*}
\PPer_\sigma (\tilde E,B_r) &=&
\Per (\tilde E,B_{r+\Upsilon})\\
&=& \Per (\tilde E,B_{r}) +\Per (\tilde E,B_{r+\Upsilon}\setminus B_r)\\
&=& r^{2-\sigma} \,\Per_\sigma (\tilde E,\,B_1)+ 4\Upsilon.
\end{eqnarray*}
This proves~\eqref{612:0012}.

{F}rom~\eqref{612:0012} we obtain that
\begin{equation}\label{D4600} {\mathcal{E}}_{B_r} (\tilde u,\tilde E) \ge
\Phi\Big( r^{2-\sigma} \,\Per_\sigma (\tilde E,\,B_1)+\vartheta\Big).\end{equation}
On the other hand, recalling~\eqref{WD6:PO},
we have that
\begin{equation}\label{8djj6yuUU:PRE}
\Per_\sigma (E_r,B_r)=
\Per_\sigma (r E_\sharp,B_r) = r^{2-\sigma}\Per_\sigma(E_\sharp,B_1)\le
r^{2-\sigma}\,\big( \Per_\sigma(\tilde E,B_1)-a\big).\end{equation}
Now we claim that
\begin{equation}\label{8djj6yuUU}
\PPer_\sigma (E_r,B_r)
\le r^{2-\sigma}\,\big( \Per_\sigma(\tilde E,B_1)-a\big)+\vartheta.\end{equation}
Indeed, if~$\sigma\in(0,1)$ then \eqref{8djj6yuUU} reduces to~\eqref{8djj6yuUU:PRE}.
If instead~$\sigma=1$ we use the fact that~$E_r$ coincides with~$\tilde E$
outside~$B_r$ and~\eqref{8djj6yuUU:PRE} to see that
\begin{eqnarray*}
\PPer_\sigma (E_r,B_r) &=& \Per(E_r,B_{r+\Upsilon})\\
&=& \Per(E_r,B_{r})+ \Per(E_r,B_{r+\Upsilon}\setminus B_r)
\\ &\le& r^{2-\sigma}\,\big( \Per_\sigma(\tilde E,B_1)-a\big)+4\Upsilon.
\end{eqnarray*}
This establishes~\eqref{8djj6yuUU}.

Then, the monotonicity of~$\Phi$ and \eqref{8djj6yuUU}
give that
\begin{equation}\label{GFj56111:a}
\Phi\Big( \PPer_\sigma (E_r,B_r)\Big)
\le \Phi\Big( r^{2-\sigma}\,\big( \Per_\sigma(\tilde E,B_1)-a\big)+\vartheta\Big).\end{equation}
Now we remark that  
$$ |\nabla u_r(X)|\le |\nabla \tilde u(X)\,\psi(X/r)|
+r^{-1}  |\tilde u(X)\,\nabla\psi(X/r)|
\le |X| + C r^{-1} |X|^2,$$
for some~$C>0$.
In consequence of this, and possibly renaming~$C>0$, we obtain
$$ \int_{B_r} |\nabla u_r(X)|^2\,dX
\le C \int_{B_r} \Big( |X|^2 + r^{-2} |X|^4\Big)\,dX \le Cr^4.$$
This and~\eqref{GFj56111:a}
give that
$$ {\mathcal{E}}_{B_r}(u_r,E_r)
\le Cr^4 + \Phi\Big( r^{2-\sigma}\,\big( \Per_\sigma(\tilde E,B_1)-a\big)+\vartheta\Big).$$
Putting together this, \eqref{D4600:0} and~\eqref{D4600}, we
conclude that
$$ \Phi\Big( r^{2-\sigma} \,\Per_\sigma (\tilde E,\,B_1)+\vartheta\Big)
\le Cr^4 + 
\Phi\Big( r^{2-\sigma}\,\big( \Per_\sigma(\tilde E,B_1)-a\big)+\vartheta\Big).$$
Thus, if~$r^{2-\sigma} \,\Per_\sigma (\tilde E,\,B_1)\le \frac{1}{2}$,
and so~$\Per_\sigma (\tilde E,\,B_1)+\vartheta\le1$,
we can use~\eqref{QDV78}
and obtain
\begin{equation}\label{LA65GT8}
\Big[ r^{2-\sigma} \,\Per_\sigma (\tilde E,\,B_1)+\vartheta \Big]^\gamma
\le Cr^4 +\Big[
r^{2-\sigma}\,\big( \Per_\sigma(\tilde E,B_1)-a\big)+\vartheta\Big]^\gamma.\end{equation}
Now we distinguish the cases~$\sigma\in(0,1)$ and~$\sigma=1$.
When~$\sigma\in(0,1)$ then~$\vartheta=0$ and so~\eqref{LA65GT8}
becomes
$$ r^{(2-\sigma)\gamma} \,\Big(\Per_\sigma (\tilde E,\,B_1)\Big)^\gamma
\le Cr^4 +
r^{(2-\sigma)\gamma}\,\Big( \Per_\sigma(\tilde E,B_1)-a\Big)^\gamma.$$
So we multiply by~$r^{(\sigma-2)\gamma}$ and we get
$$ a_*:=\Big(\Per_\sigma (\tilde E,\,B_1)\Big)^\gamma
-
\Big( \Per_\sigma(\tilde E,B_1)-a\Big)^\gamma \le Cr^{4+(\sigma-2)\gamma} .$$
Notice that~$a_*>0$ since so is~$a$,
and therefore the latter inequality gives a contradiction if~$r$
is small enough, thanks to~\eqref{gamma4}. This concludes the case
in which~$\sigma\in(0,1)$.

If instead~$\sigma=1$, then we have that~$\vartheta>0$ and so, for small~$t$,
we have that
$$ (t+\vartheta)^\gamma = \vartheta^\gamma +\gamma \vartheta^{\gamma-1} t + O(t^2).$$
Therefore, we infer from~\eqref{LA65GT8} that
$$ \vartheta^\gamma +\gamma \vartheta^{\gamma-1} \,
\, r^{2-\sigma} \,\Per_\sigma (\tilde E,\,B_1)
\le \vartheta^\gamma +\gamma \vartheta^{\gamma-1} \,
\,r^{2-\sigma}\,\big( \Per_\sigma(\tilde E,B_1)-a\big)+O(r^{4-2\sigma}).$$
Hence we simplify some terms and we divide by~$r^{2-\sigma}$,
to obtain
$$ a\le O(r^{2-\sigma}),$$
which gives a contradiction for small~$r>0$. This completes also the case~$\sigma=1$.
\end{proof}

\subsection{Completion of the proof of Theorem~\ref{EXAMPLE}}

The claim in Theorem~\ref{EXAMPLE} now follows
plainly by combining Corollary~\ref{LARGE BALLS}
and Proposition~\ref{SMALL BALLS}.

\section{Proof of Theorem~\ref{FBW}}\label{sec:equa}

The argument is a combination of a classical domain variation
(see e.g.~\cite{alt}) with an expansion of the (classical or
fractional) perimeter. 
Some similar perturbative methods appear, in the classical case,
for instance in~\cites{GL, CKL}.
Since the arguments
involved here use both standard and non-standard
observations, we give all the details for the facility of the reader.
First, we observe that
\begin{equation}\label{H5ed66}
\begin{split}
& {\mbox{the function }}\,
\Xi:=
\big( \partial_{\nu}^+ u (x)\big)^2
-\big( \partial_{\nu}^- u (x)\big)^2
- H_{\sigma}^E(x)\,\Phi'\Big(
\PPer_\sigma (E,\Omega)
\Big) \\ &\qquad{\mbox{ belongs to }}\; C(\partial E\cap\Omega),\end{split}\end{equation}
thanks to~\eqref{II51}, \eqref{II52} and
Proposition~6.3 in~\cite{I5} (to be used when~$\sigma\in(0,1)$).

Also, given a vector field~$V\in C^\infty(\R^n,\R^n)$
such that
\begin{equation}\label{7hhh}
{\mbox{$V(x)=0$ for any~$x\in\Omega^c$,}}\end{equation}
for small~$t\in\R$ we consider the ODE flow~$y=y(t;x)$ given
by the Cauchy problem
\begin{equation}\label{7hhh:2}
\left\{
\begin{matrix}
\partial_t y(t;x) = V(y(t;x)),\\
y(0;x)= x.
\end{matrix}
\right.\end{equation}
We remark that, for small~$t\in\R$,
\begin{equation}\label{8j234}
\begin{split}
y(t;x) \;&= x + t \,V(y(t;x)) + o(t)
\\ &= x + t \,V(x) + o(t). 
\end{split}
\end{equation}
Accordingly,
\begin{equation}\label{7yhtf66}
\begin{split}
D_x y(t;x) \;&= I + t \,DV(x) + o(t)
\\ &= I + t \,DV(y(t;x)) + o(t),
\end{split}\end{equation}
where~$I$ denoted the $n$-dimensional identity matrix.

Also, the map~$\R^n\ni x\mapsto y(t;x)$
is invertible for small~$t$, i.e. we can consider
the inverse diffeomorphism~$x(t;y)$. In this way,
\begin{equation}\label{2JK:JH56}
x\big( t;\,y(t;x)\big)=x \;{\mbox{ and }}\;
y\big( t;\,x(t;x)\big)=y
.\end{equation}
By~\eqref{8j234}, we know that
\begin{equation}\label{4.6bis}\begin{split}
x(t;y) =\;& y\big(t; x(t;y)\big)-t\,V\big(
y\big(t; x(t;y)\big)\big)+o(t) \\
=\;& y-t\,V(y)+o(t),\end{split}\end{equation}
and therefore
$$ D_y x(t;y) = I - t \,DV(y) + o(t).$$
In particular,
\begin{equation} \label{DET:8s}
\det D_y x(t;y) = 1-t\,{\rm{div}}\, V (y)+o(t).\end{equation}
Now, given a minimal pair~$(u,E)$ as in the statement
of Theorem~\ref{FBW}, we define
$$ u_t (y):= u( x(t;y) ).$$
We remark that the subscript~$t$ here above
does not represent a time derivative. By~\eqref{2JK:JH56},
we can write~$u(x)=u_t(y(t;x))$ and thus,
recalling~\eqref{7yhtf66},
\begin{equation}\label{7uh7765a4}
\begin{split} 
\nabla u(x)\;& = D_x y(t;x)\,\nabla u_t(y(t;x))\\
&= \nabla u_t(y(t;x)) + t\,DV(y(t;x))\,\nabla u_t(y(t;x)) +o(t).
\end{split}\end{equation}
Also, we consider the image of the set~$E$ under the
diffeomorphism~$y(t;\cdot)$, i.e. we define
$$ E_t := y(t;E).$$
We claim that
\begin{equation}\label{ADM:7t4}
{\mbox{the pair~$(u_t,E_t)$ is admissible.}}
\end{equation}
To check this, let~$y\in E_t$ (resp., $y\in E_t^c$).
Then there exists
\begin{equation}\label{6r5uu12}
{\mbox{$x\in E$ (resp., $x\in E^c$)}}\end{equation}
such that~$y=y(t;x)$.
Then, by~\eqref{2JK:JH56}, we have that
$$ x(t;y)=x\big( t;\,y(t;x)\big)=x.$$
This identity and~\eqref{6r5uu12} imply that
$$ 0\le u(x)=u(x(t;y))=u_t (y)
\qquad{\mbox{ (resp., $0\ge u_t(y)$).}}$$
{F}rom this, we obtain~\eqref{ADM:7t4}.

In addition, we recall that
\begin{equation}\label{uj7123}
{\mbox{$y(t;x)=x$
for any~$x\in\Omega^c$,}}\end{equation} thanks to~\eqref{7hhh}
and~\eqref{7hhh:2}. Therefore, we have that
\begin{equation}\label{uj7123:OMEGA}
y(t;\Omega)=\Omega.\end{equation}
Moreover, as a consequence of~\eqref{uj7123}
and of~\eqref{ADM:7t4}, and using the minimality of~$(u,E)$,
we have that
\begin{equation}\label{8uGGHK}
0\le {\mathcal{E}}_\Omega(u_t,E_t)-
{\mathcal{E}}_\Omega(u,E).
\end{equation}
Now we compute the first order in~$t$ of
the right hand side of~\eqref{8uGGHK}.
For this scope,
using, for instance,
formula~(6.3) (when~$\sigma=1$) or formula~(6.12) (when~$\sigma\in(0,1)$)
in~\cite{I5}, and recalling that~$V$ vanishes outside~$\Omega$,
one obtains that
\begin{equation}\label{HJ:6541sd}
\PPer_\sigma(E_t,\Omega)=\PPer_\sigma(E,\Omega)
+ t\int_{(\partial E)\cap\Omega}
H_{\sigma}^E(x)\, V(x)\cdot\nu(x)\,d{\mathcal{H}}^{n-1}(x) 
+o(t).\end{equation}
Here above, we denoted by~$\nu$
the exterior normal of~$E$ and by~${\mathcal{H}}^{n-1}$
the $(n-1)$-dimensional Hausdorff measure.

{F}rom~\eqref{HJ:6541sd}, we obtain that
\begin{equation}\label{HJ:6541sd:NEW}
\begin{split}
&\Phi\Big(\PPer_\sigma(E_t,\Omega)\Big)
= \Phi\left(
\PPer_\sigma(E,\Omega)
+t\int_{(\partial E)\cap\Omega}
H_{\sigma}^E(x)\, V(x)\cdot\nu(x)\,d{\mathcal{H}}^{n-1}(x) 
+o(t)\right) \\
&\quad=
\Phi\Big(
\PPer_\sigma(E,\Omega)\Big)
+t\,\Phi'\Big(
\PPer_\sigma(E,\Omega)\Big)
\,
\int_{(\partial E)\cap\Omega}
H_{\sigma}^E(x)\, V(x)\cdot\nu(x)\,d{\mathcal{H}}^{n-1}(x)
+o(t).
\end{split}
\end{equation}
Moreover, by~\eqref{7uh7765a4},
$$ |\nabla u(x)|^2= 
|\nabla u_t(y(t;x))|^2 + 2t\,
\nabla u_t(y(t;x))\cdot \Big(
DV(y(t;x))\,\nabla u_t(y(t;x)) \Big) +o(t). $$
Now we integrate this equation in~$x$ over~$\Omega$
and we use the change of variable~$y:=y(t;x)$.
In this way, recalling~\eqref{DET:8s}
and~\eqref{uj7123:OMEGA}, we see that
\begin{eqnarray*}
&&\int_\Omega |\nabla u(x)|^2\,dx \\ &=&
\int_\Omega \Big[ 
|\nabla u_t(y(t;x))|^2 + 2t\,
\nabla u_t(y(t;x))\cdot \Big(
DV(y(t;x))\,\nabla u_t(y(t;x)) \Big) \Big]\,dx +o(t)\\
&=&
\int_\Omega \Big[
|\nabla u_t(y)|^2 + 2t\,
\nabla u_t(y)\cdot \Big(
DV(y)\,\nabla u_t(y) \Big) \Big]\,
|\det D_y x(t;y)|\,dy +o(t)
\\ &=& 
\int_\Omega \Big[
|\nabla u_t(y)|^2 + 2t\,
\nabla u_t(y)\cdot \Big(
DV(y)\,\nabla u_t(y) \Big) \Big]\,\big[
1-t\,{\rm{div}}\, V (y)\big]\,dy+o(t)\\
&=&
\int_\Omega \Big[
|\nabla u_t(y)|^2 + 2t\,
\nabla u_t(y)\cdot \Big(
DV(y)\,\nabla u_t(y) \Big) -t\,|\nabla u_t(y)|^2\,{\rm{div}}\, V (y)\Big]
\,dy+o(t).\end{eqnarray*}
We write this formula as
\begin{equation}\label{YGFuaH}
\begin{split}
&\int_\Omega |\nabla u_t(y)|^2\,dy =
\int_\Omega |\nabla u(x)|^2\,dx
\\&\qquad+ t\,
\int_\Omega \Big[
|\nabla u_t(y)|^2\,{\rm{div}}\, V (y)
-2\,\nabla u_t(y)\cdot \Big(
DV(y)\,\nabla u_t(y) \Big)
\Big]\,dy+o(t).
\end{split}\end{equation}
Also, by~\eqref{7uh7765a4},
$$ \nabla u(x) = \nabla u_t(y(t;x)) + O(t),$$
and so, evaluating this expression at~$x:=x(t;y)$ and using~\eqref{4.6bis},
we get
$$ \nabla u_t(y)=\nabla u_t\big(y\big(t;x(t;y)\big)\big)=
\nabla u(x(t;y))+O(t)
=\nabla u(y)+O(t).$$
We can substitute this into~\eqref{YGFuaH}, thus obtaining
\begin{equation}\label{YGFuaH:0}
\begin{split}
&\int_\Omega |\nabla u_t(y)|^2\,dy =
\int_\Omega |\nabla u(x)|^2\,dx
\\&\qquad+ t\,
\int_\Omega \Big[
|\nabla u(y)|^2\,{\rm{div}}\, V (y)
-2\,\nabla u(y)\cdot \Big(
DV(y)\,\nabla u(y) \Big)
\Big]\,dy+o(t).
\end{split}\end{equation}
Now we define~$\Omega_1:=\Omega\cap \{u>0\}$ and~$\Omega_2:=\Omega\cap \{u<0\}$.
Notice that~$\Delta u=0$ in~$\Omega_1$ and in~$\Omega_2$, thanks to
Lemma~\ref{LE:AR}. Accordingly, 
in both~$\Omega_1$ and~$\Omega_2$
we have that
\begin{equation}\label{90:9ijh1}
{\rm{div}}\,\big( |\nabla u|^2\,V\big) =
|\nabla u|^2 {\rm{div}}\,V + 2V\cdot \big( D^2 u\,\nabla u\big)\end{equation}
and
\begin{equation}\label{90:9ijh2}
{\rm{div}}\,\big( (V\cdot\nabla u)\nabla u\big)
= \nabla (V\cdot\nabla u)\cdot \nabla u
= \nabla u\cdot (DV \nabla u) +V\cdot \big(D^2 u\,\nabla u\big).\end{equation}
So, we take the quantity in~\eqref{90:9ijh1}
and we subtract twice the quantity in~\eqref{90:9ijh2}:
in this way we see that, in both~$\Omega_1$ and~$\Omega_2$,
\begin{eqnarray*}
&& {\rm{div}}\,\big( |\nabla u|^2\,V\big)- 
2{\rm{div}}\,\big( (V\cdot\nabla u)\nabla u\big)\\
&=& 
|\nabla u|^2 {\rm{div}}\,V + 2V\cdot \big( D^2 u\,\nabla u\big)
-2\Big[
\nabla u\cdot (DV \nabla u) +V\cdot \big(D^2 u\,\nabla u\big)\Big]\\
&=& |\nabla u|^2 {\rm{div}}\,V -2\nabla u\cdot (DV \nabla u).
\end{eqnarray*}
We remark that the last expression is exactly the
quantity appearing in one integrand of~\eqref{YGFuaH:0}:
therefore we can write~\eqref{YGFuaH:0} as
\begin{equation}\label{YGFuaH:0:BIS}
\begin{split}
&\int_\Omega |\nabla u_t(y)|^2\,dy =
\int_\Omega |\nabla u(x)|^2\,dx
\\&\qquad+ t\,\sum_{ i\in\{1,2\} }
\int_{\Omega_i} \Big[
{\rm{div}}\,\big( |\nabla u(y)|^2\,V(y)\big)- 
2{\rm{div}}\,\big( \big(V(y)\cdot\nabla u(y)\big)\nabla u(y)\big)
\Big]\,dy+o(t).
\end{split}\end{equation}
Now we recall~\eqref{7u994a5}
and we notice that the exterior normal~$\nu_1$ of~$\Omega_1$ coincides
with~$\nu$, while the exterior normal~$\nu_2$
of~$\Omega_2$ coincides
with~$-\nu$. 
Furthermore, by~\eqref{II53}, 
we see that $\nu_1=-\frac{\nabla u}{|\nabla u|} = 
-\frac{\nabla u}{|\partial_{\nu}^+ u|}$ coming from~$\Omega_1$
and~$\nu_2=\frac{\nabla u}{|\nabla u|}=
\frac{\nabla u}{|\partial_{\nu}^- u|}$ coming from~$\Omega_2$.
Accordingly, coming from~$\Omega_1$, we have that
$$ \partial_{\nu_1} u =\nu_1\cdot \nabla u
= -\frac{\nabla u}{|\nabla u|}\cdot\nabla u
= - |\partial_{\nu}^+u|.$$
Similarly, coming from~$\Omega_2$,
$$ \partial_{\nu_2} u =\nu_2\cdot \nabla u=
\frac{\nabla u}{|\nabla u|}\cdot\nabla u
= |\partial_{\nu}^-u|.$$
Therefore, coming from~$\Omega_1$
$$ \nabla u \;\partial_{\nu_1} u = -|\nabla u|\;\partial_{\nu_1} u\;\nu_1
=|\partial_{\nu}^+u|^2 \;\nu,$$
and coming from~$\Omega_2$
$$ \nabla u\;\partial_{\nu_2} u = |\nabla u|\;\partial_{\nu_2} u\;\nu_2
=-|\partial_{\nu}^- u|^2 \;\nu.$$
Consequently, coming from~$\Omega_1$ we have that
$$ |\nabla u|^2\,V\cdot\nu_1
-2 (V\cdot\nabla u)\partial_{\nu_1} u
=|\partial_{\nu}^+u|^2 V\cdot\nu
- 2(V\cdot\nu)\,
|\partial_{\nu}^+u|^2 
= -|\partial_{\nu}^+u|^2 V\cdot\nu,$$
while, coming from~$\Omega_2$,
$$ |\nabla u|^2\,V\cdot\nu_2
-2 (V\cdot\nabla u)\partial_{\nu_2} u
= -|\partial_{\nu}^- u|^2 V\cdot\nu
+2 (V\cdot\nu)\,|\partial_{\nu}^- u|^2=
|\partial_{\nu}^- u|^2 V\cdot\nu.$$
Hence, if we apply the Divergence Theorem in~\eqref{YGFuaH:0:BIS},
we obtain
\begin{equation}\label{YGFuaH:0:TRIS}
\begin{split}
&\int_\Omega |\nabla u_t(y)|^2\,dy -\int_\Omega |\nabla u(x)|^2\,dx
\\=\;&t\,\sum_{ i\in\{1,2\} }
\int_{\partial\Omega_i} 
\Big[|\nabla u(y)|^2\,V(y)\cdot\nu_i(y)
-2 \big(V(y)\cdot\nabla u(y)\big)\partial_{\nu_i} u(y)
\Big]\,d{\mathcal{H}}^{n-1}(y)+o(t)
\\=\;&-t\,\int_{(\partial E)\cap\Omega}
|\partial_{\nu}^+u(y)|^2 \,V(y)\cdot\nu(y)\,d{\mathcal{H}}^{n-1}(y)\\
&\qquad\qquad+t\,\int_{(\partial E)\cap\Omega}
|\partial_{\nu}^-u(y)|^2 \,V(y)\cdot\nu(y)
\,d{\mathcal{H}}^{n-1}(y)+
o(t).
\end{split}\end{equation}
Using this and~\eqref{HJ:6541sd:NEW},
and also recalling the definition in~\eqref{H5ed66},
we conclude that
\begin{eqnarray*}
&& {\mathcal{E}}_\Omega(u_t,E_t)-{\mathcal{E}}_\Omega(u,E)\\
&=& \int_\Omega |\nabla u_t(y)|^2\,dy -\int_\Omega |\nabla u(x)|^2\,dx
+\Phi\Big(\PPer_\sigma(E_t,\Omega)\Big)-\Phi\Big(\PPer_\sigma(E,\Omega)\Big)\\
&=&
t\,\int_{(\partial E)\cap\Omega}
\Big( |\partial_{\nu}^-u(y)|^2 
-|\partial_{\nu}^+u(y)|^2\Big)
\,V(y)\cdot\nu(y)
\,d{\mathcal{H}}^{n-1}(y)
\\ &&\qquad+
t\,\Phi'\Big(
\PPer_\sigma(E,\Omega)\Big)
\,
\int_{(\partial E)\cap\Omega}
H_{\sigma}^E(x)\, V(x)\cdot\nu(x)\,d{\mathcal{H}}^{n-1}(x)
+o(t)\\
&=& -t\int_{(\partial E)\cap\Omega}
\Xi(x)\,V(x)\cdot\nu(x)\,d{\mathcal{H}}^{n-1}(x)
+o(t).
\end{eqnarray*}
This and~\eqref{8uGGHK} imply that
$$ \int_{(\partial E)\cap\Omega}
\Xi(x)\,V(x)\cdot\nu(x)\,d{\mathcal{H}}^{n-1}(x)
=0.$$
Since~$V$ is arbitrary, the latter identity
and~\eqref{H5ed66} imply that~$\Xi$ vanishes in the whole
of~$\partial E\cap\Omega$, which completes the proof of
Theorem~\ref{FBW}.

\section{Proof of Theorem~\ref{GFB}}\label{sec:growth}

\subsection{Energy of the harmonic replacement of
a minimal solutions}

We start with a computation on the harmonic replacement:

\begin{lemma}\label{HT:65:1}
Assume that~\eqref{LIPASS} holds true.
Let~$(u,E)$ be a minimal pair in~$\Omega$, with~$u\ge0$ a.e. in~$\Omega^c$
and~$B_{R_o}\Subset\Omega$.
Let~$R\in(0,R_o]$ and~$u_R$ be the function minimizing the Dirichlet energy in~$B_R$
among all the functions~$v$ such that~$v-u\in H^1_0(B_R)$. Then
$$ \int_{B_R} |\nabla u(x)-\nabla u_R(x)|^2\,dx
\le C\,L_Q\,R^{n-\sigma},$$
for some~$C>0$, possibly depending on~$R_o$, $n$ and~$\sigma$, and~$L_Q$
is the one introduced in~\eqref{LIPASS}.
\end{lemma}

\begin{proof} We observe that~$u\ge0$ a.e. in~$\R^n$, thanks to
Lemma~\ref{MAX:PLE}. Hence~$u_R\ge0$ a.e., by the classical maximum
principle, and therefore, taking~$u_R:=u$ in~$B_R^c$, we see that~$(u_R,E\cup B_R)$
is an admissible pair, and an admissible competitor against~$(u,E)$.
Therefore, by the minimality of~$(u,E)$,
\begin{equation} \label{789912df}
\begin{split}&
0\le {\mathcal{E}}_\Omega(u_R,E\cup B_R)-
{\mathcal{E}}_\Omega(u,E)
\\ &\quad= \int_{B_R} \big(|\nabla u_R(x)|^2-|\nabla u(x)|^2\big)\,dx
+ \Phi\Big( \PPer_\sigma(E\cup B_R,\Omega)\Big)
-\Phi\Big( \PPer_\sigma(E,\Omega)\Big) .\end{split}\end{equation}
Now we use the
subadditivity of the (either classical or fractional)
perimeter (see e.g.
Proposition~3.38(d) in~\cite{AFP} when~$\sigma=1$
and formula~(3.1) in~\cite{DFPV} when~$\sigma\in(0,1)$) and we
remark that 
\begin{equation}\label{7s13jq1}
\begin{split}& \PPer_\sigma(E\cup B_R,\Omega)\le
\PPer_\sigma(E,\Omega) +
\PPer_\sigma(B_R,\Omega)\le \PPer_\sigma(E,\Omega)+
\Per_\sigma(B_R,\R^n)\\ &\qquad= \PPer_\sigma(E,\Omega)+
R^{n-\sigma}
\Per_\sigma(B_1,\R^n) \le Q,\end{split}\end{equation}
in light of~\eqref{LIBO}.

Now we claim that
\begin{equation}\label{8usHH71}
\Phi\Big( \PPer_\sigma(E\cup B_R,\Omega)\Big)
-\Phi\Big( \PPer_\sigma(E,\Omega)\Big)
\le C\,L_Q\,R^{n-\sigma}.
\end{equation}
To prove it, we observe that if $\PPer_\sigma(E\cup B_R,\Omega)
\le\PPer_\sigma(E,\Omega)$ then, by the monotonicity of~$\Phi$
it follows that~$\Phi\Big( \PPer_\sigma(E\cup B_R,\Omega)\Big)
\le\Phi\Big( \PPer_\sigma(E,\Omega)\Big)$, which implies~\eqref{8usHH71}.
Therefore, we can assume that~$\PPer_\sigma(E\cup B_R,\Omega)
>\PPer_\sigma(E,\Omega)$. Then, by~\eqref{LIPASS}
(which can be utilized here in view of~\eqref{7s13jq1}), and using again the
subadditivity of the (either classical or fractional)
perimeter,
\begin{eqnarray*}&& \Phi\Big( \PPer_\sigma(E\cup B_R,\Omega)\Big)
-\Phi\Big( \PPer_\sigma(E,\Omega)\Big)\le L_Q\,\Big|
\PPer_\sigma(E\cup B_R,\Omega) -\PPer_\sigma(E,\Omega)
\Big|
\\ &&\qquad\le L_Q\, \PPer_\sigma(B_R,\Omega)
\le L_Q\, \Per_\sigma(B_R,\R^n)
\le C\,L_Q\,R^{n-\sigma}.\end{eqnarray*}
This proves~\eqref{8usHH71}.

By \eqref{8usHH71} and~\eqref{789912df} we obtain
\begin{eqnarray*}
C\,L_Q\,R^{n-\sigma} &\ge &
\int_{B_R} \big(|\nabla u(x)|^2-|\nabla u_R(x)|^2\big)\,dx\\
&=&
\int_{B_R} \big(\nabla u(x)+\nabla u_R(x)\big)
\cdot \big(\nabla u(x)-\nabla u_R(x)\big)\,dx
\\ &=&
\int_{B_R} \big(\nabla u(x)-\nabla u_R(x)+2\nabla u_R(x)\big)
\cdot \big(\nabla u(x)-\nabla u_R(x)\big)\,dx
\\ &=&
\int_{B_R} \big|\nabla u(x)-\nabla u_R(x)\big|^2\,dx
+2\int_{B_R} \nabla u_R(x)
\cdot \big(\nabla u(x)-\nabla u_R(x)\big)\,dx
\\ &=&\int_{B_R} \big|\nabla u(x)-\nabla u_R(x)\big|^2\,dx,
\end{eqnarray*}
where the latter equality follows from the fact that~$u_R$ is harmonic in~$B_R$.
The desired result is thus established.
\end{proof}

\begin{rem}
{F}rom Lemma~\ref{HT:65:1} 
it follows that the gradient of
the minimizers locally
belongs to the Campanato space~${\mathcal{L}}^{p,\lambda}$,
with~$p:=2$ and~$\lambda:=n-\sigma$, and thus to the Morrey
space~$L^{2,n-\sigma}$. This and the Poincar\'e inequality
would give that the minimizers belong to the
Campanato space~${\mathcal{L}}^{2,n+2-\sigma}$,
and thus to the H\"older space of continuous
functions with exponents~$\frac{(n+2-\sigma)-n}{2}=1-\frac\sigma2$.
In any case, in the forthcoming Section~\ref{sec:coro} 
we will provide an alternate approach to continuity
results.\end{rem}

\subsection{Estimate on the average of minimal solutions}

Now we estimate the average in balls for minimal solutions:

\begin{lemma}\label{GH:78hj}
Assume that~\eqref{LIPASS} holds true.
Let~$(u,E)$ be a minimal pair in~$\Omega$, with~$u\ge0$ a.e. in~$\Omega^c$
and~$B_{R_o}(p)\Subset\Omega$. Assume that~$R\in(0,R_o]$
and~$p\in{\mathcal{U}}_0$.
Then
$$ \frac{1}{|B_R(p)|} \int_{B_R(p)} u (x)\,dx
\le C\,\sqrt{L_Q}\, R^{1-\frac{\sigma}{2}},$$
for some~$C>0$, possibly depending on~$R_o$, $n$ and~$\sigma$, and~$L_Q$
is the one introduced in~\eqref{LIPASS}.
\end{lemma}

\begin{proof} By~\eqref{DEF:0}, we can take a sequence~$p_k$ with
\begin{equation}\label{PSILIM0}
\lim_{k\to+\infty} u(p_k)=0.
\end{equation}
For any~$r\in(0,R]$ and for any~$k\in\N$, we define
$$ \psi(r):= r^{-n} \int_{B_r(p)} u(x) \,dx
\quad{\mbox{ and }}\quad
\psi_k(r):= r^{-n} \int_{B_r(p_k)} u(x) \,dx
.$$
We observe that
\begin{equation}\label{PSILIM}
\lim_{k\to+\infty} \psi_k(r)=\psi(r).
\end{equation}
To check this, we let~$\bar R> R_o$, with~$B_{\bar R}(p)\Subset\Omega$
and we consider a continuous approximation of~$u$ in~$L^1(B_{\bar R}(p))$.
That is, we take continuous functions~$u_\epsilon$ such that
\begin{equation}\label{in epso} 
\lim_{\epsilon\searrow0} \int_{ B_{\bar R}(p) } \big|u(x)-u_\epsilon(x)\big|\,dx=0.\end{equation}
For large~$k$, we have that~$B_r(p_k)\subseteq B_{\bar R}(p)$, and so
\begin{eqnarray*}&&
r^n\,|\psi_k(r)-\psi(r)|\\ &=&
\left| \int_{B_r(p_k)} u(x) \,dx-\int_{B_r(p)} u(x) \,dx\right|\\
&\le&
\left| \int_{B_r(p_k)} u_\epsilon(x) \,dx-\int_{B_r(p)} u_\epsilon(x) \,dx\right|
+2\int_{ B_{\bar R}(p) } \big|u(x)-u_\epsilon(x)\big|\,dx\\
&=&
\left| \int_{B_r} \big( u_\epsilon(x+p_k)-u_\epsilon(x+p) \big)\,dx\right|
+2\int_{ B_{\bar R}(p) } \big|u(x)-u_\epsilon(x)\big|\,dx.
\end{eqnarray*}
Hence, taking the limit in~$k$ and using the Dominated Convergence Theorem,
we get that
\[
\lim_{k\to+\infty} r^n\,|\psi_k(r)-\psi(r)|\le
2\int_{ B_{\bar R}(p) } \big|u(x)-u_\epsilon(x)\big|\,dx.
\]
Then, we take the limit in~$\epsilon$ and we obtain~\eqref{PSILIM}
from~\eqref{in epso}, as desired.

Now, we recall that~$u\ge0$ a.e. in~$\R^n$, thanks to Lemma~\ref{MAX:PLE}.
Thus, by Remark~\ref{RE:su},
\begin{equation}\label{psi:0}
\psi_k(0):=\lim_{r\searrow0} \psi_k(r)=u(p_k) .\end{equation}
Furthermore, using polar coordinates,
\begin{equation}\label{8i67qart7}
\begin{split}
& \psi_k'(r) = \frac{d}{dr} \int_{B_1} u(p_k+ry)\,dy
= \int_{B_1} \nabla u(p_k+ry)\cdot y\,dy \\
&\qquad = \int_0^1 \left[ t^n
\int_{S^{n-1}} \nabla u(p_k+rt\omega)\cdot \omega\,d{\mathcal{H}}^{n-1}(\omega)\right]\,dt=
\int_0^1 \left[ t^n
\int_{\partial B_1} \partial_\nu u(p_k+rt\omega)
\,d{\mathcal{H}}^{n-1}(\omega)\right]\,dt,
\end{split}\end{equation}
where~$\nu$ is the exterior normal of~$B_1$.

Now, fixed~$k\in\N$,
we use the notation of Lemma~\ref{HT:65:1} for
the harmonic replacement~$u_r$ in~$B_r(p_k)\Subset\Omega$. 
For~$\rho\in(0,r]$,
we define~$v_r(x):= u_r(p_k+\rho x)$ and we observe that,
for any~$x\in B_1$, we have~$\Delta v_r(x)= \rho^2\Delta u_r(p_k+\rho x)
=0$, and so
$$ 0=\int_{B_1} \Delta v_r(x)\,dx
=\int_{\partial B_1} \partial_\nu v_r(\omega)\,d{\mathcal{H}}^{n-1}(\omega)
= \rho 
\int_{\partial B_1} \partial_\nu u_r(p_k+\rho\omega)\,d{\mathcal{H}}^{n-1}(\omega).$$
We take~$\rho:=rt$ and we insert this into~\eqref{8i67qart7}.
In this way, we obtain
$$ \psi_k'(r) = \int_0^1 \left[ t^n
\int_{\partial B_1} \Big( \partial_\nu u(p_k+rt\omega)-
\partial_\nu u_{r}(p_k+rt\omega)\Big)
\,d{\mathcal{H}}^{n-1}(\omega)\right]\,dt.$$
That is, using polar coordinate backwards and making the change
of variable~$y:=p_k+rx$,
\begin{eqnarray*}
\psi'_k(r)& =& \int_{B_1} x\cdot
\Big(\nabla u(p_k+rx)-
\nabla u_{r}(p_k+rx)\Big)\,dx\\
&=& r^{-(n+1)}
\int_{B_r(p_k)} (y-p_k)\cdot
\Big(\nabla u(y)-
\nabla u_{r}(y)\Big)\,dy.\end{eqnarray*}
Hence, using the H\"older Inequality and Lemma~\ref{HT:65:1},
$$ \psi'_k(r)\le 
r^{-n}
\int_{B_r(p_k)} \big|\nabla u(y)-
\nabla u_{r}(y)\big|\,dy\le
C\,r^{-\frac{n}{2}}
\sqrt{ \int_{B_r(p_k)} \big|\nabla u(y)-
\nabla u_{r}(y)\big|^2\,dy }
\le C \,\sqrt{L_Q}\; r^{-\frac{\sigma}{2}},$$
for some~$C>0$.
This and~\eqref{psi:0} give that
$$ \psi_k(R)-u(p_k)=\psi_k(R)-\psi_k(0)=\int_0^R \psi'_k(r)\,dr\le
C \,\sqrt{L_Q}\,\int_0^R r^{-\frac{\sigma}{2}}\le
C\,\sqrt{L_Q}\, R^{1-\frac{\sigma}{2}},$$
up to renaming constants.
Hence, making use of~\eqref{PSILIM0}
and~\eqref{PSILIM}, we find that
$$ \psi(R)\le
C\,\sqrt{L_Q}\, R^{1-\frac{\sigma}{2}},$$
that is the desired claim.\end{proof}

\subsection{Completion of the proof of Theorem~\ref{GFB}}

We recall that~$u\ge0$ a.e. in~$\R^n$, thanks to
Lemma~\ref{MAX:PLE}. 
In particular, $u$ is subharmonic, thanks to Lemma~\ref{HJ:AR678a},
and thus
\begin{equation}\label{KJ:78JJ1a2}
\frac{1}{|B_\rho|}\int_{B_{\rho}(x)} u (y)\,dy\ge u(x),
\end{equation}
for small~$\rho>0$.
Now we take~$x\in\Omega$, with~$|x|$ suitably small,
and we define~$R:=|x|$. Notice that~$B_{R}(x)\subseteq B_{2R}$ and therefore, since~$u\ge0$,
\begin{equation}\label{KJ:78JJ1a2:BIS}
\int_{B_{R}(x)} u (y)\,dy\le \int_{B_{2R}} u (y)\,dy.
\end{equation}
In addition, by applying Lemma~\ref{GH:78hj} in $B_{2R}$, we find that
$$ \frac{1}{R^n} \int_{B_{2R}} u (y)\,dy
\le C\,\sqrt{L_Q}\, R^{1-\frac{\sigma}{2}}.$$
As a result, exploiting~\eqref{KJ:78JJ1a2} and~\eqref{KJ:78JJ1a2:BIS},
\begin{eqnarray*}
u(x)\le \frac{C}{R^n}
\int_{B_{R}(x)} u (y)\,dy\le
\frac{C}{R^n}
\int_{B_{2R}} u (y)\,dy\le
C\,\sqrt{L_Q}\, R^{1-\frac{\sigma}{2}}=C\,\sqrt{L_Q}\, |x|^{1-\frac{\sigma}{2}},
\end{eqnarray*}
up to renaming constants. This proves
Theorem~\ref{GFB}.

\section{Proof of Corollary~\ref{COCO}}\label{sec:coro}

To prove Corollary~\ref{COCO}, it is useful to point out a strengthening of
Lemma~\ref{LE:AR}, in which one replaces the condition on the infimum
with a pointwise condition (this refinement is possible in virtue of
Theorem~\ref{GFB}):

\begin{lemma}\label{LE:AR:BIS}
Let the assumptions of Corollary~\ref{COCO} hold true.
Let~$(u,E)$ be a minimal pair in~$\Omega$, with~$u\ge0$. Let~$U\Subset\Omega$ be an open set
with~$u>0$ in~$U$. Then~$u$ is harmonic
in~$U$.
\end{lemma}

\begin{proof} Let~$U'\Subset U$ be open. The claim is proved if we show that~$u$ is
harmonic in~$U'$.
To this aim,
we claim that
\begin{equation}\label{INFBUO}
\inf_{U'} u>0.
\end{equation}
We argue for a contradiction, assuming that this infimum is equal to~$0$.
Then, recalling~\eqref{DEF:0}, we have that there exists~$x_\star\in\overline{U'}\cap
{\mathcal{U}}_0$. In particular, since~$x_\star\in \overline{U'}\subset U$,
we know that
\begin{equation}\label{1CFA12394}
u(x_\star)>0.
\end{equation}
On the other hand, by Theorem~\ref{GFB}, for small~$y$,
$$ u(x_\star+y)\le 
C\,\sqrt{L_Q}\, |y|^{1-\frac{\sigma}{2}}.$$
As a result, recalling Remark~\ref{RE:su},
$$ u(x_\star)=u^+(x_\star)=
\lim_{\epsilon\searrow0}
\frac{1}{|B_\epsilon|} \int_{B_\epsilon } u^+(x_\star+y)\,dy
\le C\,\sqrt{L_Q}\,\lim_{\epsilon\searrow0}
\frac{1}{|B_\epsilon|} \int_{B_\epsilon } |y|^{1-\frac{\sigma}{2}}\,dy=0.$$
This is in contradiction with~\eqref{1CFA12394}
and so we have proved~\eqref{INFBUO}.

Then, in light of~\eqref{INFBUO}, we fall under the assumptions
of Lemma~\ref{LE:AR}, which in turn implies the desired claim.
\end{proof}

First we recall that~$u\ge0$ a.e. in~$\R^n$, thanks to
Lemma~\ref{MAX:PLE}.
Also we know that~$u$
is subharmonic in~$\Omega$ (recall Lemma~\ref{HJ:AR678a})
and therefore, by the classical maximum principle,
\begin{equation}\label{usu8}
u(x)\le M\end{equation}
for any~$x\in\Omega$.
Also, we may suppose that
\begin{equation}\label{usqu567aa8}
{\mbox{there exists~$q_o\in B_{3R/10}$ such that~$u(q_o)=0$.}}
\end{equation}
Indeed, if this does not hold, then~$u$ is harmonic in~$B_{3R/10}$,
due to Lemma~\ref{LE:AR:BIS}, and thus
$$ \sup_{B_{R/4}} |\nabla u|\le \frac{C}{R} \sup_{B_{3R/10}} u
\le \frac{CM}{R},$$
for some~$C>0$, where we also used~\eqref{usu8} in the latter inequality.
This implies that
$$ |u(x)-u(y)|\le \frac{CM}{R}\,|x-y|\le
\frac{CM}{R^{1-\frac\sigma2}} \,|x-y|^{1-\frac\sigma2},$$
which gives the desired result in this case.

Hence, from now on, we can suppose that~\eqref{usqu567aa8} holds true.
We fix~$x\ne y\in B_{R/4}$ and we define~$d(x)$ (resp.~$d(y)$)
to be the distance from~$x$ (resp. from~$y$) to the set~$\{u=0\}$.
By~\eqref{usqu567aa8}, we know that~$d(x)$, $d(y)\in [0,\,3R/5]$.
We distinguish two cases:
\begin{itemize}
\item[{\it Case 1}:] $|x-y|\ge \displaystyle\frac{\max\{d(x),\,d(y)\}}{2}$,
\item[{\it Case 2}:] $|x-y|<\displaystyle\frac{\max\{d(x),\,d(y)\}}{2}$.
\end{itemize}
First, we deal with {\it Case~1}.
In this case, we use Theorem~\ref{GFB}
and we have that
$$ |u(x)|\le C\,\sqrt{L_Q}\, (d(x))^{1-\frac{\sigma}{2}}
\; {\mbox{and }}\;
|u(y)|\le C\,\sqrt{L_Q}\, (d(y))^{1-\frac{\sigma}{2}}.$$
As a consequence,
$$ |u(x)-u(y)|\le |u(x)|+|u(y)|
\le  C\,\sqrt{L_Q}\, \Big( (d(x))^{1-\frac{\sigma}{2}}
+(d(y))^{1-\frac{\sigma}{2}}\Big).$$
Then, the assumption of {\it Case~1} implies
$$ |u(x)-u(y)|\le C\,\sqrt{L_Q}\,|x-y|^{1-\frac{\sigma}{2}},$$
up to renaming constants, which gives the desired result in this case.

Now we consider {\it Case~2}. In this case, up to exchanging $x$ and~$y$,
we have that
\begin{equation} \label{0HY:A}
0\le2|x-y|<d(x)=\max\{d(x),\,d(y)\}\end{equation}
and~$u>0$ in~$B_{d(x)}(x)$. Then, by Lemma~\ref{LE:AR:BIS},
we know that~$u$ is harmonic in~$B_{d(x)}(x)$ and thus
\begin{equation}\label{JH:jonbo}
\sup_{B_{9d(x)/10}(x)} |\nabla u|\le \frac{C}{d(x)}\,\sup_{B_{d(x)}(x)} u,\end{equation}
for some~$C>0$. 

Now, we prove that
\begin{equation}\label{JH:jonbo:2}
\sup_{B_{d(x)}(x)} u\le C\,\sqrt{L_Q}\,\big(d(x)\big)^{1-\frac{\sigma}{2}},
\end{equation}
for some~$C>0$.
For this, take~$\eta\in {B_{d(x)}(x)}$.
By construction, there exists~$\zeta\in \overline{B_{d(x)}(x)}$
such that~$u(\zeta)=0$. Accordingly, we have that~$|\eta-\zeta|\le
|\eta-x|+|x-\zeta|\le 2d(x)$, and then, by Theorem~\ref{GFB},
$$ u(\eta) \le  C\,\sqrt{L_Q}\, |\eta-\zeta|^{1-\frac{\sigma}{2}}
\le C\,\sqrt{L_Q}\,\big( d(x)\big)^{1-\frac{\sigma}{2}},$$
up to renaming~$C>0$, and this establishes~\eqref{JH:jonbo:2}.

Thus, exploiting~\eqref{JH:jonbo}
and~\eqref{JH:jonbo:2}, and possibly renaming constants, we obtain that
$$ \sup_{B_{9d(x)/10}(x)} |\nabla u|\le C\,\sqrt{L_Q}\,\big( d(x)\big)^{-\frac{\sigma}{2}}.$$
Notice now that~$y \in B_{d(x)/2}(x)\subset B_{9d(x)/10}(x)$, thanks to~\eqref{0HY:A},
therefore
$$ |u(x)-u(y)|\le C\,\sqrt{L_Q}\,\big( d(x)\big)^{-\frac{\sigma}{2}}\,
|x-y|\le 
C\,\sqrt{L_Q}\,|x-y|^{1-\frac{\sigma}{2}},$$
up to renaming constant. This establishes the desired result also in {\it Case~2}
and so the proof of
Corollary~\ref{COCO} is now completed.

\section{Proof of Theorem~\ref{DENS}}\label{sec:dens}

The proof is based on a measure theoretic argument
that was used, in different forms, in~\cites{CSV, serena1P},
but
differently from the proof in the existing literature,
we cannot use here the scaling properties of the functional:
namely, the existing proofs can always reduce to the unit
ball, since the rescaled minimal pair is a minimal
pair for the rescaled functional, while this procedure
fails in our case (as stressed for instance by Theorem~\ref{EXAMPLE}).
For this reason, we need to perform a measure theoretic
argument which works at every scale. To this goal,
for any~$r\in(0,R)$ we define
$$ V(r) := |B_r\setminus E| \;{\mbox{ and }}\;
a(r):={\mathcal{H}}^{n-1} \big( (\partial B_r)\setminus E\big)$$
and we observe that
\begin{equation}\label{8u5tUUa}
V(r)=
\int_0^r a(t)\,dt,\end{equation}
see e.g. formula (13.3) in~\cite{maggi}.

The proof of Theorem~\ref{DENS} is by contradiction:
we suppose that, for some~$r_o \in (0,R/2)$, we have
that
\begin{equation}\label{6t56YUB7uI}
V(r_o) = |B_{r_o}\setminus E|\le \delta r_o^n
\end{equation}
and we derive a contradiction if~$\delta>0$ is sufficiently small.
We recall that~$u\ge0$ a.e. in~$\R^n$, due to Lemma~\ref{MAX:PLE},
and we define
$$ A:= B_{r}\setminus E.$$
We observe that~$(u, E\cup A)$ is admissible,
since~$(E\cup A)^c = E^c \cap A^c \subseteq E^c$.
Then, by the minimality of~$(u,E)$, we obtain that
\begin{equation}\label{7HBFttg}
\begin{split}
0\;&\le {\mathcal{E}}_\Omega(u,E\cup A)-{\mathcal{E}}_\Omega(u,E)
\\ &= \Phi\Big( \PPer_\sigma (E\cup A,\Omega)\Big)
- \Phi\Big( \PPer_\sigma (E,\Omega)\Big).
\end{split}\end{equation}
Now, by the subadditivity of the (either classical or fractional)
perimeter (see e.g.
Proposition~3.38(d) in~\cite{AFP} when~$\sigma=1$ 
and formula~(3.1) in~\cite{DFPV} when~$\sigma\in(0,1)$), we have that
\begin{eqnarray*}
\PPer_\sigma (E\cup A,\Omega) &=& \PPer_\sigma (E\cup B_{r},\Omega)\\
&\le&
\PPer_\sigma (E,\Omega) + \PPer_\sigma (B_{r},\Omega)
\\ &\le&
\PPer_\sigma (E,\Omega) +
\Per_\sigma (B_{r},\R^n)
\\ &\le& \PPer_\sigma (E,\Omega) + R^{n-\sigma} \Per_\sigma (B_{1},\R^n).
\end{eqnarray*}
Then, both~$\PPer_\sigma (E,\Omega)$ and~$\PPer_\sigma (E\cup A,\Omega)$
are bounded by~$P$, as defined in~\eqref{12sKJ:9yuhh:P}
and so lie in the invertibility range of~$\Phi$, as prescribed by~\eqref{12sKJ:9yuhh}.
This observation and~\eqref{7HBFttg} imply that
\begin{equation}\label{FaD6fdsgha:PRE}
\PPer_\sigma (E,\Omega)\le \PPer_\sigma (E\cup A,\Omega).\end{equation}
Now we claim that
\begin{equation}\label{FaD6fdsgha} 
\Per_\sigma (E,\Omega)\le \Per_\sigma (E\cup A,\Omega).\end{equation}
Indeed, when~$\sigma\in(0,1)$ then~\eqref{FaD6fdsgha} is simply~\eqref{FaD6fdsgha:PRE}.
If instead~$\sigma=1$ we notice that~$E\setminus \overline{B_r}
=(E\cup A)\setminus \overline{B_r}$ and so we use~\eqref{ssJK:1qs:BIS},
\eqref{ssJK:1qs:BIS:2} and~\eqref{FaD6fdsgha:PRE} to obtain that
\begin{eqnarray*}
&& 0\le \PPer_\sigma (E\cup A,\Omega)-\PPer_\sigma (E,\Omega)
= \Per_\sigma (E\cup A,\overline{B_r})-\Per_\sigma (E,\overline{B_r})
\\ &&\qquad\qquad=\Per_\sigma (E\cup A,\Omega)-\Per_\sigma (E,\Omega),\end{eqnarray*}
which establishes~\eqref{FaD6fdsgha}.

Now we use the (either classical or fractional) isoperimetric inequality in the
whole of~$\R^n$ (see e.g. Theorem~3.46 in~\cite{AFP}
when~$\sigma=1$, and~\cite{frank}, or
Corollary~25 in~\cite{CVCV} when~$\sigma\in(0,1)$): in this way,
we have that
\begin{equation}\label{2525} \big( V(r)\big)^{\frac{n-\sigma}{n}}
= |B_r\setminus E|^{\frac{n-\sigma}{n}} = |A|^{\frac{n-\sigma}{n}}
\le C\,\Per_\sigma (A,\R^n),\end{equation}
for some~$C>0$.

Now we claim that, for a.e.~$r\in(0,R)$,
\begin{equation}\label{JJKJ:9jty}
\Per_\sigma (A,\R^n)\le \left\{
\begin{matrix}
C\, a(r) & {\mbox{ if }} \sigma=1,\\
C\,\displaystyle\int_0^r a(\rho) \,(r-\rho)^{-\sigma}\,d\rho
& {\mbox{ if }} \sigma\in(0,1),
\end{matrix}
\right.
\end{equation}
for some~$C>0$ (up to renaming~$C$).
First we prove~\eqref{JJKJ:9jty} when~$\sigma=1$.
For this, we 
write the perimeter of~$E$ in term of the
Gauss-Green measure~$\mu_E$ (see Remark 12.2 in~\cite{maggi}),
we use the additivity of the measures on disjoint sets
and we obtain that
\begin{equation}\label{J7:1}
\begin{split}
&\Per(E,B_r)+\Per(E,\Omega\setminus\overline{B_r})=
|\mu_E| (B_r) + |\mu_E|(\Omega\setminus\overline{B_r}) \\
&\qquad \le |\mu_E| (B_r) + |\mu_E|(\Omega\setminus{B_r})
=|\mu_E|(\Omega)=\Per(E,\Omega).
\end{split}
\end{equation}
Now we prove that, for a.e.~$r\in(0,R)$, we have
\begin{equation}\label{BGT7} 
{\mathcal{H}}^{n-1}\big( (\partial B_r)\setminus E\big)
= \Per(B_r\setminus E,\Omega)-\Per(E,B_r).\end{equation}
For this scope,
we make use of the property of the 
Gauss-Green measure with
respect to
the intersection with balls (see formula~(15.14)
in Lemma 15.12 of~\cite{maggi}, applied here to the complement of~$E$).
In this way, we see that
\begin{eqnarray*}
{\mathcal{H}}^{n-1}\big( (\partial B_r)\setminus E\big)
&=&{\mathcal{H}}^{n-1}\big( (\partial B_r)\cap E^c\cap\Omega\big)\\
&=&
{\mathcal{H}}^{n-1}\Big|_{ E^c \cap (\partial B_r)} (\Omega)
\\ &=&|\mu_{E^c\cap B_r}|(\Omega) - |\mu_{E^c}|\Big|_{B_r}(\Omega)\\
&=& \Per(E^c\cap B_r,\Omega) - |\mu_{E^c}| (B_r\cap\Omega)
\\&=& \Per(E^c\cap B_r,\Omega) - |\mu_{E^c}| (B_r)
\\&=&\Per(E^c\cap B_r,\Omega)-\Per(E^c,B_r).
\end{eqnarray*}
{F}rom this and the fact that~$\Per(E^c,B_r)=\Per(E,B_r)$
(see for instance
Proposition~3.38(d) in~\cite{AFP}),
we obtain that~\eqref{BGT7}
holds true.

Now we claim that, for a.e.~$r\in(0,R)$, we have
\begin{equation}\label{JR:MA}
\Per(E\cup B_r, \overline{B_r}) = 
{\mathcal{H}}^{n-1} \big( (\partial B_r)\setminus E\big).
\end{equation}
Since it is not easy to find a complete reference for
such formula in the literature, we try to give here an exhaustive proof.
To this goal, given a set~$F$ and~$t\in[0,1]$, we denote by~$F^{(t)}$ the set
of points of density~$t$ of~$F$ (see e.g. Example~5.17 in~\cite{maggi}), that is
$$ F^{(t)} := \left\{ x\in\R^n {\mbox{ s.t. }}
\lim_{r\to0}\frac{|F\cap B_r(x)|}{|B_r|}=t\right\}.$$
With this notation, we observe that~$B_r^{(0)} = \R^n\setminus\overline{B_r}$,
and thus
\begin{equation}\label{9hYIooP}
B_r^{(0)}\cap \overline{B_r}=\varnothing.
\end{equation}
We denote by~$\partial^*$ the
reduced boundary of
a set of locally finite perimeter (see e.g.
formula~(15.1) in~\cite{maggi}): we recall that for any~$x\in\partial^*E$
one can define the measure-theoretic outer unit normal to~$E$,
that we denote by~$\nu_E$.
We also recall that, by 
De Giorgi's Structure Theorem (see e.g. formula~(15.10) in~\cite{maggi}),
\begin{equation}\label{DG56HJJFG}
|\mu_E| ={\mathcal{H}}^{n-1}\Big|_{\partial^* E}. 
\end{equation}
We also set
$$ N_r := \{ x\in (\partial^* E)\cap(\partial B_r) {\mbox{ s.t. }}
\nu_E=\nu_{B_r} \}.$$
We claim that, for a.e.~$r\in(0,R)$,
\begin{equation}\label{34we3}
{\mathcal{H}}^{n-1} (N_r)=0.
\end{equation}
To check this, for any~$k\in\N$ we define
$$ \beta_k := 
\left\{ r\in(0,R) {\mbox{ s.t. }} {\mathcal{H}}^{n-1} (N_r) 
\ge \frac1k\right\}.$$
Then, if~$r\in\beta_k$, by~\eqref{DG56HJJFG}
we have that
$$ |\mu_E|(\partial B_r)
={\mathcal{H}}^{n-1}\Big|_{\partial^* E} (\partial B_r)
={\mathcal{H}}^{n-1}\big( (\partial^* E)\cap(\partial B_r)\big)
\ge {\mathcal{H}}^{n-1}(N_r) \ge\frac{1}{k}.$$
As a consequence, if~$r_1,\dots,r_j\in \beta_k$ and~$r\in(0,R)$, we obtain that
$$ \Per(E,B_R)=|\mu_E|(B_R)\ge |\mu_E|\left( \bigcup_{i=1}^j (\partial B_{r_i})\right)=
\sum_{i=1}^j |\mu_E|(\partial B_{r_i})\ge \frac{j}{k},$$
that is~$j\le k\,\Per(E,B_R)$.

This says that~$\beta_k$ has a finite (indeed less then~$k\,\Per(E,B_R)$)
number of elements. Thus the following set is countable (and so
of zero measure):
$$ \bigcup_{k=1}^{+\infty} \beta_k=
\left\{ r\in(0,R) {\mbox{ s.t. }} {\mathcal{H}}^{n-1} (N_r) 
>0\right\} = \left\{ r\in(0,R) {\mbox{ s.t. \eqref{34we3} does not hold}} \right\} .$$
This proves~\eqref{34we3}.

Now we use the known formula about the perimeter of the union.
For instance, exploiting formula~(16.12) of~\cite{maggi}
(used here with~$F=B_r$ and~$G:=\overline{B_r}$) we have that
$$ \Per(E\cup B_r,\overline{B_r})=
\Per(E, B_r^{(0)} \cap \overline{B_r})+
\Per(B_r , E^{(0)}\cap \overline{B_r})+
{\mathcal{H}}^{n-1}(N_r\cap \overline{B_r}).$$
In particular, using~\eqref{9hYIooP} and~\eqref{34we3}, we obtain that
\begin{equation}\label{H7GG:AHIAaa}
\Per(E\cup B_r,\overline{B_r})=\Per(B_r , E^{(0)}\cap \overline{B_r}),\end{equation}for a.e.~$r\in(0,R)$.
On the other hand, $B_r$ is a smooth set and so (see e.g. Example~12.6
in~\cite{maggi}) we have that
$$ \Per(B_r , E^{(0)}\cap \overline{B_r})=
{\mathcal{H}}^{n-1} \big( E^{(0)}\cap \overline{B_r}\cap(\partial B_r)\big)
= {\mathcal{H}}^{n-1} \big( E^{(0)}\cap(\partial B_r)\big),$$
and so \eqref{H7GG:AHIAaa} becomes
\begin{equation}\label{JK:1qd6shshs:aa} \Per(E\cup B_r,\overline{B_r})=
{\mathcal{H}}^{n-1} \big( E^{(0)}\cap(\partial B_r)\big).\end{equation}
Now we set
$$S:=\big( E^{(0)}\setminus E^c\big)\cup
\big( E^c\setminus E^{(0)}\big)$$ and we
remark that~$|S|=0$ (see e.g. formula~(5.19) in~\cite{maggi}).
Then, also~$|S\cap B_r|=0$.
Therefore (see e.g. Remark 12.4 in~\cite{maggi}) we get that~$\Per(S,\R^n)=0
=\Per(S\cap B_r,\R^n)$ and then
(see e.g. formula~(15.15) in~\cite{maggi})
for a.e.~$r\in(0,R)$ we obtain
$$ {\mathcal{H}}^{n-1}\big(S\cap (\partial B_r)\big)=
\Per(S\cap B_r,\R^n)-\Per(S,B_r) =0$$
and so, as a consequence, 
$$ {\mathcal{H}}^{n-1} \big( E^{(0)}\cap(\partial B_r)\big)
={\mathcal{H}}^{n-1} \big( E^{c}\cap(\partial B_r)\big).$$
Now we combine this and~\eqref{JK:1qd6shshs:aa}
and we finally complete the proof of~\eqref{JR:MA}.

Now we show that, for a.e.~$r\in(0,R)$,
\begin{equation}\label{J7:3}
\Per(E\cup B_r,\Omega) -\Per(E,\Omega\setminus\overline{B_r})=
\Per (B_r\setminus E,\Omega) -\Per(E,B_r).
\end{equation}
To prove this, we notice that~$(E\cup B_r)\setminus\overline{B_r}
=E\setminus\overline{B_r}$, and so we
use Lemma~\ref{HJ:LEM1} to see that
$$ \Per(E\cup B_r,\Omega) -\Per(E,\Omega) =
\Per(E\cup B_r,\overline{B_r})-\Per(E,\overline{B_r}).$$
As a consequence,
\begin{eqnarray*}
&& \Per(E\cup B_r,\Omega) -\Per(E,\Omega\setminus\overline{B_r})\\
&=& \Per(E\cup B_r,\overline{B_r})-\Per(E,\overline{B_r})+\Per(E,\Omega)
-\Per(E,\Omega\setminus\overline{B_r})\\
&=& \Per(E\cup B_r,\overline{B_r}) -|\mu_E|(\overline{B_r})
+|\mu_E|(\Omega)-|\mu_E|(\Omega\setminus\overline{B_r})\\
&=& \Per(E\cup B_r,\overline{B_r}),
\end{eqnarray*}
thanks to the additivity of the 
Gauss-Green measure~$\mu_E$. Then, we use~\eqref{JR:MA}
and we obtain that
$$ \Per(E\cup B_r,\Omega) -\Per(E,\Omega\setminus\overline{B_r})
={\mathcal{H}}^{n-1} \big( (\partial B_r)\setminus E\big).$$
Then, we exploit~\eqref{BGT7}
and we complete the proof of~\eqref{J7:3}.

Now we observe that, using~\eqref{BGT7}
and~\eqref{J7:3}, we obtain that, for a.e.~$r\in(0,R)$,
\begin{equation}\label{J7:2}
\Per (E\cup B_r,\Omega) =\Per (E,\Omega\setminus\overline{B_r})+
{\mathcal{H}}^{n-1}\big( (\partial B_r)\setminus E\big).
\end{equation}
Now, putting together~\eqref{J7:1} and~\eqref{J7:2}, and noticing
that~$E\cup B_r=E\cup A$, we have that
\begin{eqnarray*}
\Per(E,B_r)
&\le& \Per(E,\Omega) - \Per(E,\Omega\setminus\overline{B_r})\\
&=& \Per(E,\Omega) - \Per (E\cup B_r,\Omega) +
{\mathcal{H}}^{n-1}\big( (\partial B_r)\setminus E\big)\\&=&
\Per(E,\Omega) - \Per (E\cup A,\Omega) +
{\mathcal{H}}^{n-1}\big( (\partial B_r)\setminus E\big)
.\end{eqnarray*}
Therefore, recalling~\eqref{FaD6fdsgha} (used here with~$\sigma=1$),
we conclude that
\begin{equation} \label{LA2}
\Per(E,B_r)\le
{\mathcal{H}}^{n-1}\big( (\partial B_r)\setminus E\big).\end{equation}
Now we take~$r'\in(r,R)$ and we observe that~$B_r\Subset B_{r'}
\Subset \Omega$. Also, we see that~$
A\setminus\overline{B_{r'}}=\varnothing$,
thus, by Lemma~\ref{HJ:LEM1} (applied here with~$F:=\varnothing$),
$$ \Per(A,\R^n) =
\Per(A,\overline{B_{r'}})\le \Per(A,\Omega) =
\Per(B_r\setminus E,\Omega).$$
As a consequence of this and of~\eqref{J7:3}, we obtain
$$ \Per(A,\R^n)\le
\Per(E\cup B_r,\Omega) -\Per(E,\Omega\setminus\overline{B_r})
+\Per(E,B_r).$$
Hence, in light of~\eqref{J7:2} and~\eqref{LA2},
$$ \Per(A,\R^n)\le 2\,
{\mathcal{H}}^{n-1}\big( (\partial B_r)\setminus E\big) =2a(r).$$
This completes the proof of~\eqref{JJKJ:9jty} when~$\sigma=1$.

When~$\sigma\in(0,1)$, to prove~\eqref{JJKJ:9jty}
we use a modification of the argument contained in formulas~(5.8)--(5.12)
in~\cite{serena1P}. We first observe that 
$$ \Per_\sigma(E, \Omega)-\Per_\sigma(E\cup A, \Omega) = 
L(A,E) - L\big(A, (E\cup A)^c\big).$$
As a consequence, 
\begin{eqnarray*}
&& \Per_\sigma(A, \R^n) = L(A,A^c) = L(A,E) + L\big(A, (E\cup A)^c\big)\\
&&\qquad = 2 L\big(A, (E\cup A)^c\big) + \Per_\sigma(E, \Omega)-\Per_\sigma(E\cup A, \Omega). 
\end{eqnarray*}
This and~\eqref{FaD6fdsgha} give that 
\begin{equation}\label{agg234}
\Per_\sigma(A, \R^n)\le 2 L\big(A, (E\cup A)^c\big)\le 2 L(A, B_r^c).
\end{equation}
Now we recall that~$A\subseteq B_r$ and so, using the change
of coordinates~$\zeta:=x-y$, we obtain that
\begin{equation}\label{8.14:BIS}
\begin{split}
& L(A, B_r^c)
=\int_{A\times B_r^c}\frac{dx\,dy}{|x-y|^{n+\sigma}}
\le \int_{\{(x,\zeta)\in A\times\R^n {\mbox{ s.t. }}|\zeta|\ge r-|x|\} }
\frac{dx\,d\zeta}{|\zeta|^{n+\sigma}}\\
&\qquad\le C\,\int_A \left[\int_{r-|x|}^{+\infty} \frac{\rho^{n-1}\,d\rho}{
\rho^{n+\sigma}}
\right]\,dx
\le C\,\int_A \frac{dx}{(r-|x|)^\sigma}.\end{split}\end{equation}
Now we use the Coarea Formula (see e.g. Theorem~2
on page~117 of~\cite{EG}, applied here in codimension~$1$
to the functions~$f(x)=|x|$ and~$g(x):=\frac{\chi_A(x)}{(r-|x|)^\sigma}$), 
and we deduce that
\begin{eqnarray*}
&& \int_A \frac{dx}{(r-|x|)^\sigma} =\int_{\R}
\left[ \int_{\partial B_t} \frac{\chi_A(x)}{(r-|x|)^\sigma}
\,d{\mathcal{H}}^{n-1}(x) \right] \,dt
\\&&\qquad=\int_{0}^r
\left[ \int_{\partial B_t} \frac{\chi_{E^c}(x)}{(r-t)^\sigma}
\,d{\mathcal{H}}^{n-1}(x) \right] \,dt
= \int_{0}^r
\frac{ {\mathcal{H}}^{n-1} \big( E^c\cap(\partial B_t)\big)}{(r-t)^\sigma}\,dt
=\int_0^r \frac{a(t)}{(r-t)^\sigma}\,dt.\end{eqnarray*}
This and~\eqref{8.14:BIS} imply that
$$ L(A, B_r^c)\le C\,\int_0^r \frac{a(t)}{(r-t)^\sigma}\,dt.$$
Inserting this into~\eqref{agg234} we get 
$$ \Per_\sigma(A, \R^n)\le C\,\int_0^r \frac{a(t)}{(r-t)^\sigma}\,dt,$$
which gives the desired claim in~\eqref{JJKJ:9jty} when~$\sigma\in(0,1)$.

Using~\eqref{2525} and~\eqref{JJKJ:9jty}, and possibly renaming
constants, we conclude that, for a.e.~$r\in(0,R)$,
\begin{equation}\label{2626}
\big( V(r)\big)^{\frac{n-\sigma}{n}}
\le \left\{
\begin{matrix}
C\, a(r) & {\mbox{ if }} \sigma=1,\\
C\,\displaystyle\int_0^r a(\rho) \,(r-\rho)^{-\sigma}\,d\rho
& {\mbox{ if }} \sigma\in(0,1).
\end{matrix}
\right.
\end{equation}
Our next goal is to show that, for any~$t\in
\left[\frac14,\;\frac12\right]$,
we have that
\begin{equation}\label{2727}
\int_{r_o/4}^{tr_o} \big( V(r)\big)^{\frac{n-\sigma}{n}}\,dr
\le C t^{1-\sigma} r_o^{1-\sigma} \, V(tr_o),
\end{equation}
for some~$C>0$.
To prove this, we integrate~\eqref{2626}
in~$r\in \left[\frac{r_o}{4},\;tr_o\right]$.
Then, when~$\sigma=1$, we obtain~\eqref{2727}
directly from~\eqref{8u5tUUa}. If instead~$\sigma\in(0,1)$,
we obtain
\begin{eqnarray*}
&& \int_{r_o/4}^{tr_o} \big( V(r)\big)^{\frac{n-\sigma}{n}}\,dr
\le C \int_{r_o/4}^{tr_o}\left[
\int_0^r a(\rho) \,(r-\rho)^{-\sigma}\,d\rho\right]\,dr\\
&&\qquad \le C
\int_{0}^{tr_o}\left[
\int_\rho^{tr_o} a(\rho) \,(r-\rho)^{-\sigma}\,dr\right]\,d\rho
=\frac{C}{1-\sigma}\,
\int_{0}^{tr_o} a(\rho) (tr_o-\rho)^{1-\sigma} \,d\rho\\
&&\qquad \le \frac{C}{1-\sigma}\,
\int_{0}^{tr_o} a(\rho) (tr_o)^{1-\sigma} \,d\rho
= \frac{C\,(tr_o)^{1-\sigma}}{1-\sigma}\,V(tr_o),
\end{eqnarray*}
where we used~\eqref{8u5tUUa} in the last identity.
This completes the proof of~\eqref{2727},
up to renaming the constants.

Now we define~$t_k :=\frac14+\frac{1}{2^k}$, for any~$k\ge2$.
Let also~$w_k:= r_o^{-n}\,V(t_k r_o)$. Notice that~$t_{k+1}\ge1/4$.
Then we use~\eqref{2727}
with~$t:=t_k$ and we obtain that
$$ C t_k^{1-\sigma} r_o^{1-\sigma} \, V(t_k r_o)\ge
\int_{r_o/4}^{t_k r_o} \big( V(r)\big)^{\frac{n-\sigma}{n}}\,dr
\ge \int_{t_{k+1}r_o}^{t_k r_o} \big( V(r)\big)^{\frac{n-\sigma}{n}}\,dr.$$
Thus, since~$V(\cdot)$ is monotone,
$$ C t_k^{1-\sigma} r_o^{1-\sigma}\, V(t_k r_o)\ge
\big( t_kr_o- t_{k+1}r_o\big)\,
\big( V(t_{k+1}r_o)\big)^{\frac{n-\sigma}{n}}= 
\frac{r_o}{2^{k+1}}\,\big( V(t_{k+1}r_o)\big)^{\frac{n-\sigma}{n}}.$$
This can be written as
$$ w_{k+1}^{\frac{n-\sigma}{n}} =
r_o^{\sigma-n}\,\big( V(t_{k+1} r_o)\big)^{\frac{n-\sigma}{n}}
\le 2^{k+1}\,C\,t_k^{1-\sigma} r_o^{-n} V(t_k r_o)=
2^{k+1}\,C\,t_k^{1-\sigma} w_k.$$
Consequently, using that~$t_k\le1$ and
possibly renaming~$C>0$, we obtain that
\begin{equation}\label{q146JK8yt}
w_{k+1}^{\frac{n-\sigma}{n}}\le C^k w_k.\end{equation}
Also, we have that~$t_2=\frac12$ and thus
$$ w_2 =r_o^{-n} V\left(\frac{r_o}{2}\right)\le r_o^{-n} V(r_o)\le\delta,$$
in view of~\eqref{6t56YUB7uI}.
Then, if~$\delta>0$ is sufficiently small,
we have that~$w_k \to0$ as~$k\to+\infty$ (see e.g. formula~(8.18)
in~\cite{alessio} for explicit bounds). This and the fact that~$t_k\ge \frac14$
say that
$$ 0 =\lim_{k\to+\infty} r_o^{-n}\,V(t_k r_o)
=\lim_{k\to+\infty} r_o^{-n}\, |B_{t_kr_o}\setminus E|
\ge r_o^{-n}\, |B_{r_o/4}\setminus E|.$$
Hence, we have that~$|B_{r_o/4}\setminus E|=0$,
in contradiction with the assumption that~$0\in\partial E$
(in the measure theoretic sense). The proof
of Theorem~\ref{DENS} is thus complete.

\section{Proof of Theorem~\ref{DENS2}}\label{sec:dens2}

By Lemma~\ref{MAX:PLE}, we have that
\begin{equation}\label{8uh5tg6}
{\mbox{$u\ge0$ a.e. in~$\R^n$. }}\end{equation}
For any~$r\in(0,R)$ we define
$$ V(r) := |B_r\cap E| \;{\mbox{ and }}\;
a(r):={\mathcal{H}}^{n-1} \big( (\partial B_r)\cap E\big)$$
and we observe that
\begin{equation}\label{8u5tUUa:BISS}
V(r)=
\int_0^r a(t)\,dt,\end{equation}
see e.g. formula (13.3) in~\cite{maggi}.

The proof of Theorem~\ref{DENS2} is obtained by a contradiction argument.
Namely, we suppose that, for some~$r_o \in (0,R/2)$ we have
that
\begin{equation}\label{6t56YUB7uI:BIS}
V(r_o) = |B_{r_o}\cap E|\le \delta_*\, r_o^n
\end{equation}
and we derive a contradiction if~$\delta_* >0$ is sufficiently small.

We let~$A:=B_r\cap E$.
Let also~$\tilde v$ be the minimizer of the Dirichlet
energy in~$B_{r_o}$ among all the possible candidates~$v:\R^n\to\R$,
such that~$v=u$ outside~$B_{r_o}$, $v-u\in H^1_0(B_{r_o})$
and~$v=0$ a.e. in~$E^c\cup A$ (for the existence
and the uniqueness of such harmonic replacement see
e.g. page~481 in~\cite{salsa}). By~\eqref{8uh5tg6}
and Lemma~2.3 in~\cite{salsa} we have that
\begin{equation}\label{hj78FD7J}
{\mbox{$\tilde v\ge0$
a.e. in~$\R^n$.}}\end{equation}
Now we set~$F:=E\setminus A$. We observe that~$\tilde v=0$ a.e. in~$F^c =E^c\cup A$
by construction. This and~\eqref{hj78FD7J} give
that~$(\tilde v,F)$ is an admissible pair, and recall also that~$\tilde v-u\in
H^1_0(B_{r_o})\subseteq H_0^1(\Omega)$.
Hence, the minimality of~$(u,E)$ gives that
\begin{eqnarray*}
0 &\le & {\mathcal{E}}_\Omega(\tilde v,F)-{\mathcal{E}}_\Omega(u,E)\\
&=& \int_\Omega |\nabla \tilde v(x)|^2\,dx
- \int_\Omega |\nabla u(x)|^2\,dx+\Phi\Big( \PPer_\sigma (F,\Omega)\Big)
-\Phi\Big( \PPer_\sigma (E,\Omega)\Big).
\end{eqnarray*}
Using this and the fact that~$\tilde v$ and~$u$ coincide
outside~$B_{r_o}$, we obtain that
\begin{equation}\label{j7842S}
\Phi\Big( \PPer_\sigma (E,\Omega)\Big) - \Phi\Big( \PPer_\sigma (F,\Omega)\Big)
\le \int_{B_{r_o}} |\nabla \tilde v(x)|^2\,dx
- \int_{B_{r_o}} |\nabla u(x)|^2\,dx.\end{equation}
Now we take~$\tilde w$ to be the minimizer of the Dirichlet energy
in~$B_{r_o}$ among all the functions~$w:\R^n\to\R$,
such that~$w=u$ outside~$B_{r_o}$, $w-u\in H^1_0(B_{r_o})$
and~$w=0$ a.e. in~$E^c$. We remark that~$u$ is a competitor with
such~$\tilde w$ and therefore
$$ \int_{B_{r_o}} |\nabla \tilde w(x)|^2\,dx
\le \int_{B_{r_o}} |\nabla u(x)|^2\,dx.$$
Plugging this into~\eqref{j7842S}, we deduce that
$$ \Phi\Big( \PPer_\sigma (E,\Omega)\Big) - \Phi\Big( \PPer_\sigma (F,\Omega)\Big)
\le \int_{B_{r_o}} |\nabla \tilde v(x)|^2\,dx
- \int_{B_{r_o}} |\nabla \tilde w(x)|^2\,dx.$$
This and Lemma~2.3 in~\cite{CSV} imply that
\begin{equation}\label{9s11234h:OPRE}
\Phi\Big( \PPer_\sigma (E,\Omega)\Big) - \Phi\Big( \PPer_\sigma (F,\Omega)\Big)
\le C\,{r_o}^{-2}\, |A|\,\|\tilde w\|_{L^\infty(B_{r_o})}^2.\end{equation}
Since, by Lemma~2.3 in~\cite{salsa},
we know that~$\tilde w\ge0$ a.e. in~$\R^n$ and is subharmonic,
we have that~$w$ in~$B_{r_o}$
takes its maximum along~$\partial B_{r_o}$,
where it coincides with~$u$. Hence
\begin{equation}\label{9s11234h}
\|\tilde w\|_{L^\infty(B_{r_o})}\le \sup_{\partial B_{r_o}} u.\end{equation}
Now we observe that condition~\eqref{JH:2345aua123}
allows us to use Theorem~\ref{GFB}, which gives that
$$ \sup_{\partial B_{r_o}} u\le C\,\sqrt{L_Q}\,r_o^{1-\frac\sigma2},$$
for some~$C>0$.
Hence~\eqref{9s11234h} gives that
$$ \|\tilde w\|_{L^\infty(B_{r_o})}\le C\,\sqrt{L_Q}\,{r_o}^{1-\frac\sigma2}.$$
Thus, recalling~\eqref{9s11234h:OPRE}, and possibly renaming constants,
we conclude that
\begin{equation}\label{9s11234h:OPRE:BIS}
\Phi\Big( \PPer_\sigma (E,\Omega)\Big) - \Phi\Big( \PPer_\sigma (F,\Omega)\Big)
\le C\,{r_o}^{-\sigma}\, |A|\,L_Q.\end{equation}
Now we claim that
\begin{equation}\label{9s11234h:OPRE:TRIS}
\Per_\sigma (E,\Omega) - \Per_\sigma (F,\Omega)
\le C\,c_o^{-1}\,{r_o}^{-\sigma}\, |A|\,L_Q,\end{equation}
where~$c_o>0$ is the one introduced in~\eqref{an2383:li:12}.
To check this, we may suppose that~$\lambda_1:=\Per_\sigma (E,\Omega) >
\Per_\sigma (F,\Omega)=:\lambda_2$, otherwise we are done. 
Then, by~\eqref{LIBO:LALL}, both~$\lambda_1$ and~$\lambda_2$ belong to~$[0,Q]$,
therefore we can make use of~\eqref{an2383:li:12} and obtain
\begin{eqnarray*}
&& \Phi\Big( \PPer_\sigma (E,\Omega)\Big) - \Phi\Big( \PPer_\sigma (F,\Omega)\Big)
=\Phi(\lambda_1)-\Phi(\lambda_2)\\
&&\qquad=\int_{\lambda_2}^{\lambda_1} \Phi'(t)\,dt \ge c_o\, (\lambda_1-\lambda_2)=
c_o\, \Big(\PPer_\sigma (E,\Omega)-\PPer_\sigma (F,\Omega)\Big)
\end{eqnarray*}
and then it follows from~\eqref{9s11234h:OPRE:BIS} that
\begin{equation} \label{9i12dyjn}
\PPer_\sigma (E,\Omega) - \PPer_\sigma (F,\Omega)
\le C\,c_o^{-1}\,{r_o}^{-\sigma}\, |A|\,L_Q.\end{equation}
Now we observe that~$E\setminus\overline{B_r}=F\setminus\overline{B_r}$,
therefore, using~\eqref{ssJK:1qs:BIS}
and~\eqref{ssJK:1qs:BIS:2}, we see that
\begin{eqnarray*} \PPer_\sigma (E,\Omega) - \PPer_\sigma (F,\Omega)
&=& \Per_\sigma (E,\overline{B_r}) - \Per_\sigma (F,\overline{B_r})\\
&=&\Per_\sigma (E,\Omega) - \Per_\sigma (F,\Omega).\end{eqnarray*}
Putting together this and~\eqref{9i12dyjn} we obtain~\eqref{9s11234h:OPRE:TRIS}.

Now we show that, for a.e.~$r\in(0,{r_o})$,
\begin{equation}\label{TO:ST71}
\Per_\sigma (A,\R^n)\le 
\left\{
\begin{matrix}
C\, \Big( a(r)+c_o^{-1}\,{r_o}^{-\sigma}\, |A|\,L_Q\Big) & {\mbox{ if }} \sigma=1,\\
C\,\left(
\displaystyle\int_0^r a(\rho) \,(r-\rho)^{-\sigma}\,d\rho
+c_o^{-1}\,{r_o}^{-\sigma}\, |A|\,L_Q\right)
& {\mbox{ if }} \sigma\in(0,1).
\end{matrix}
\right.
\end{equation}
To prove~\eqref{TO:ST71} we distinguish the cases~$\sigma=1$
and~$\sigma\in(0,1)$. If~$\sigma=1$, we notice that~$A\setminus\overline{B_r}
=(B_r\cap E)\setminus\overline{B_r}=\varnothing$, hence, by Lemma~\ref{HJ:LEM1},
we have that
$$ \Per(A,\R^n)=\Per(A,\overline{B_r})=\Per(E\cap B_r, \overline{B_r}).$$
Hence we use the formula for the perimeter
associated with the intersection with balls (see e.g.~(15.14)
in Lemma 15.12 of~\cite{maggi}) and we obtain
\begin{equation}\label{v5h6yga45}
\begin{split}
\Per(A,\R^n) \;=\;& |\mu_{E\cap B_r}|( \overline{B_r})\\
=\;&{\mathcal{H}}^{n-1}\Big|_{E\cap (\partial B_r)}( \overline{B_r})
+ |\mu_E|\Big|_{B_r}( \overline{B_r})\\
=\;& {\mathcal{H}}^{n-1}\big( E\cap (\partial B_r)\cap\overline{B_r}\big)+
\Per(E, B_r\cap \overline{B_r})\\
=\;& {\mathcal{H}}^{n-1}\big( E\cap (\partial B_r)\big)
+\Per(E, B_r).\end{split}
\end{equation}
On the other hand, we have that~$(E\setminus B_r)^c=E^c \cup B_r$,
hence (see e.g. formula~(16.11) in~\cite{maggi})
we obtain that~$\Per(E\setminus B_r, \overline{B_r})=\Per (E^c\cup B_r,
\overline{B_r})$, for a.e.~$r\in(0,{r_o})$. Hence, by Lemma~\ref{HJ:LEM1},
\begin{equation}
\begin{split}
\label{hj89:91wsdc6yGGaG}
& \Per (E,\Omega)-\Per(F,\Omega)=
\Per (E,\overline{B_r})-\Per(F,\overline{B_r})
\\&\qquad = \Per (E,\overline{B_r})-\Per(E\setminus B_r,\overline{B_r})
=\Per (E,\overline{B_r}) - \Per (E^c\cup B_r,\overline{B_r}),
\end{split}\end{equation}
for a.e.~$r\in(0,{r_o})$.
Moreover (see e.g. formula~\eqref{JR:MA}, applied here to the complementary set),
we have that
$$ \Per(E^c \cup B_r, \overline{B_r}) =
{\mathcal{H}}^{n-1} \big( (\partial B_r)\cap E\big),$$
so we can write~\eqref{hj89:91wsdc6yGGaG} as
$$ \Per (E,\Omega)-\Per(F,\Omega)=
\Per (E,\overline{B_r}) - {\mathcal{H}}^{n-1} \big( (\partial B_r)\cap E\big).$$
In particular
$$ \Per (E,{B_r})\le
\Per (E,\overline{B_r}) =
\Per (E,\Omega)-\Per(F,\Omega)+
{\mathcal{H}}^{n-1} \big( (\partial B_r)\cap E\big).$$
Then we insert this information into~\eqref{v5h6yga45}
and we obtain that
$$ \Per(A,\R^n) \le
2{\mathcal{H}}^{n-1}\big( E\cap (\partial B_r)\big)
+\Per (E,\Omega)-\Per(F,\Omega).$$
Now we recall~\eqref{9s11234h:OPRE:TRIS}
complete the proof of~\eqref{TO:ST71} when~$\sigma=1$,
and we now focus on the case~$\sigma\in(0,1)$. For this,
we use~\eqref{9g12:LA} and we see that
$$ \Per_\sigma(E,\Omega)-\Per_\sigma(F,\Omega)=
\Per_\sigma(E,\Omega)-\Per_\sigma(E\setminus A,\Omega)
=L(A,E^c)-L(A,E\setminus A).$$
Therefore
\begin{equation}\label{9s11234h:OPRE:TRIS:A}
\begin{split}
\Per_\sigma(A,\R^n)\;& =\; L(A,A^c) = L(A, E^c)+L(A, E\setminus A)\\
&=\; \Per_\sigma(E,\Omega)-\Per_\sigma(F,\Omega) +2L(A, E\setminus A).
\end{split}\end{equation}
Now we use the fact that~$A\subseteq B_r$ and the change
of coordinates~$\zeta:=x-y$ to write
\begin{equation}\label{8.14}
\begin{split}
& L(A, E\setminus A) \le L(A, B_r^c)
=\int_{A\times B_r^c}\frac{dx\,dy}{|x-y|^{n+\sigma}}
\le \int_{\{(x,\zeta)\in A\times\R^n {\mbox{ s.t. }}|\zeta|\ge r-|x|\} }
\frac{dx\,d\zeta}{|\zeta|^{n+\sigma}}\\
&\qquad\le C\,\int_A \left[\int_{r-|x|}^{+\infty} \frac{\rho^{n-1}\,d\rho}{
\rho^{n+\sigma}}
\right]\,dx
\le C\,\int_A \frac{dx}{(r-|x|)^\sigma}.\end{split}\end{equation}
Now we observe that, by Coarea Formula (see e.g. Theorem~2
on page~117 of~\cite{EG}, applied here in codimension~$1$
to the functions~$f(x)=|x|$ and~$g(x):=\frac{\chi_A(x)}{(r-|x|)^\sigma}$),
\begin{eqnarray*}
&& \int_A \frac{dx}{(r-|x|)^\sigma} =\int_{\R}
\left[ \int_{\partial B_t} \frac{\chi_A(x)}{(r-|x|)^\sigma}
\,d{\mathcal{H}}^{n-1}(x) \right] \,dt
\\&&\qquad=\int_{0}^r
\left[ \int_{\partial B_t} \frac{\chi_E(x)}{(r-t)^\sigma}
\,d{\mathcal{H}}^{n-1}(x) \right] \,dt
= \int_{0}^r
\frac{ {\mathcal{H}}^{n-1} \big( E\cap(\partial B_t)\big)}{(r-t)^\sigma}\,dt
=\int_0^r \frac{a(t)}{(r-t)^\sigma}\,dt.\end{eqnarray*}
This and~\eqref{8.14} give that
$$ L(A, E\setminus A)\le C\,\int_0^r \frac{a(t)}{(r-t)^\sigma}\,dt.$$
So we substitute this and~\eqref{9s11234h:OPRE:TRIS}
into~\eqref{9s11234h:OPRE:TRIS:A} and we complete the proof
of~\eqref{TO:ST71} when~$\sigma\in(0,1)$.

Now we recall that~$|A|=V(r)$ and
we use the (either classical or fractional) isoperimetric inequality in the
whole of~$\R^n$ (see e.g. Theorem~3.46 in~\cite{AFP}
when~$\sigma=1$, and~\cite{frank}, or
Corollary~25 in~\cite{CVCV} when~$\sigma\in(0,1)$) and we deduce from~\eqref{TO:ST71}
that, for a.e.~$r\in(0,{r_o})$,
\begin{equation}\label{JH:89:345ssjsAA}
\big(V(r)\big)^{\frac{n-\sigma}{n}}=
|A|^{\frac{n-\sigma}{n}}\le
\left\{
\begin{matrix}
C\, \Big( a(r)+c_o^{-1}\,{r_o}^{-\sigma}\, V(r)\,L_Q\Big) & {\mbox{ if }} \sigma=1,\\
C\,\left(
\displaystyle\int_0^r a(\rho) \,(r-\rho)^{-\sigma}\,d\rho
+c_o^{-1}\,{r_o}^{-\sigma}\, V(r)\,L_Q\right)
& {\mbox{ if }} \sigma\in(0,1),
\end{matrix}
\right.\end{equation}
up to renaming~$C>0$.
Now we recall~\eqref{6t56YUB7uI:BIS}
and we notice that, if~$r\in(0,r_o)$,
$$ c_o^{-1}\,{r_o}^{-\sigma}\, V(r)\,L_Q
\le c_o^{-1}\,{r_o}^{-\sigma}\, \big(
V(r)\big)^{\frac{n-\sigma}{n}}\,\big(V({r_o})\big)^{\frac{\sigma}{n}}\,L_Q\le
\delta_*^{\frac{\sigma}{n}}\,c_o^{-1}\, \big(
V(r)\big)^{\frac{n-\sigma}{n}}\,L_Q.$$
This means that, if~$\delta_*>0$ is small enough, or more precisely if
\begin{equation}\label{78:9od:cons1}
\delta_*^{\frac{\sigma}{n}}\,c_o^{-1}\,L_Q\le \frac{1}{2C},\end{equation}
we can reabsorb\footnote{It is interesting to point out that
the possibility of absorbing the term~$C\,
c_o^{-1}\,{r_o}^{-\sigma}\, V(r)\,L_Q$ into the left hand side
of~\eqref{JH:89:345ssjsAA} crucially depends on the fact
that the power produced by the (either classical or fractional) isoperimetric
inequality and the one given by the growth result
in Theorem~\ref{GFB}
match together in the appropriate way.}
one term in the left hand side of~\eqref{JH:89:345ssjsAA}:
in this way, possibly renaming constants, we obtain that, for a.e.~$r\in(0,r_o)$,
\begin{equation*}
\big(V(r)\big)^{\frac{n-\sigma}{n}}\le
\left\{
\begin{matrix}
C\, a(r) & {\mbox{ if }} \sigma=1,\\
C\,\displaystyle\int_0^r a(\rho) \,(r-\rho)^{-\sigma}\,d\rho
& {\mbox{ if }} \sigma\in(0,1).
\end{matrix}
\right.\end{equation*}
This implies that,
for any~$t\in
\left[\frac14,\;\frac12\right]$,
we have that
\begin{equation}\label{2727:999}
\int_{r_o/4}^{tr_o} \big( V(r)\big)^{\frac{n-\sigma}{n}}\,dr
\le C t^{1-\sigma} r_o^{1-\sigma} \, V(tr_o),
\end{equation}
for some~$C>0$. Indeed, the proof of~\eqref{2727:999}
is obtained as the one of~\eqref{2727}
(the only difference is that here one has to use~\eqref{8u5tUUa:BISS}
in lieu of~\eqref{8u5tUUa}). Then, one
defines~$t_k :=\frac14+\frac{1}{2^k}$ and~$w_k:= r_o^{-n}\,V(t_k r_o)$
and observes that
\begin{equation}\label{q146JK8yt:BIS}
w_{k+1}^{\frac{n-\sigma}{n}}\le C^k w_k.\end{equation}
Indeed, \eqref{q146JK8yt:BIS}
can be obtained as in the proof of~\eqref{q146JK8yt}
(but using here~\eqref{2727:999} instead of~\eqref{2727}).
Furthermore
$$ w_2= r_o^{-n}\,V\left(\frac{r_o}{2}\right)\le \delta_*,$$
thanks to~\eqref{6t56YUB7uI:BIS}. This says that, 
\begin{equation}\label{78:9od:cons2}
{\mbox{if~$\delta_*>0$
is sufficiently small (with respect to a universal constant),}}\end{equation}
then~$w_k\to0$ as $k\to+\infty$
(see formula~(8.18)
in~\cite{alessio} for explicit bounds). 
Thus
$$ 0 =\lim_{k\to+\infty} r_o^{-n}\,V(t_k r_o)
=\lim_{k\to+\infty} r_o^{-n}\, |B_{t_kr_o}\cap E|
\ge r_o^{-n}\, |B_{r_o/4}\cap E|.$$
This is in contradiction with the assumption that~$0\in\partial E$
(in the measure theoretic sense) and so the proof
of Theorem~\ref{DENS2} is finished. We stress that
the explicit condition in~\eqref{QaQ45dfg678}
comes from~\eqref{78:9od:cons1}
and~\eqref{78:9od:cons2}.

\bigskip\bigskip

\section*{Acknowledgements}

It is a pleasure to thank Francesco Maggi for an interesting discussion
and the School of Mathematics of the University of Edinburgh
for the warm hospitality. 

This work has been supported by EPSRC grant EP/K024566/1
\emph{Monotonicity formula methods for nonlinear PDEs},
Humboldt Foundation, ERC grant 277749 \emph{EPSILON Elliptic
Pde's and Symmetry of Interfaces and Layers for Odd Nonlinearities}
and PRIN grant 2012 \emph{Critical Point Theory
and Perturbative Methods for Nonlinear Differential Equations}.

\end{document}